\documentclass[10pt,reqno]{amsart}
\usepackage[utf8]{inputenc}
\usepackage[english]{babel}
\usepackage{amsmath, amssymb, amsthm}
\usepackage{mathrsfs}
\usepackage{enumitem}
\usepackage{graphicx}
\usepackage{tikz-cd}
\usepackage{xcolor}
\numberwithin{equation}{section} 

\usepackage{diagbox} 
\usepackage{tensor}  

\usepackage{hyperref} 

\usepackage[alphabetic]{amsrefs} 

\theoremstyle{plain}
\newtheorem{thm}{Theorem}[section]

\newtheorem{lemma}[thm]{Lemma}

\newtheorem{cor}[thm]{Corollary}

\theoremstyle{definition}
\newtheorem{defin}[thm]{Definition}

\newtheorem{rmk}[thm]{Remark}

\allowdisplaybreaks

\newcommand{\R}{\mathbb{R}}

\newcommand{\eps}{\varepsilon}

\newcommand{\MP}{\mathcal{MP}}
\newcommand{\A}{\mathbb{A}}

\def\XXint#1#2#3{{\setbox0=\hbox{$#1{#2#3}{\int}$ }
\vcenter{\hbox{$#2#3$ }}\kern-.6\wd0}}

\usepackage{scalerel,stackengine}
\stackMath
\newcommand\hhat[1]{%
\savestack{\tmpbox}{\stretchto{%
  \scaleto{%
    \scalerel*[\widthof{\ensuremath{#1}}]{\kern.1pt\mathchar"0362\kern.1pt}%
    {\rule{0ex}{\textheight}}
  }{\textheight}%
}{2.4ex}}%
\stackon[-6.9pt]{#1}{\tmpbox}%
}

\title[Scattering rigidity for standard stationary manifolds]{Scattering rigidity for standard stationary manifolds via timelike geodesics}
\author{Sebastián Muñoz-Thon}

\address{Department of Mathematics, Purdue University, West Lafayette, IN 47907.}
\email{smunozth@purdue.edu}

\begin{document}

\begin{abstract}
We study the scattering rigidity problem for standard stationary manifolds using timelike geodesics with a fixed momentum. Taking advantage of the symmetry of this manifolds, we use Hamiltonian reduction to show that this problem is related to scattering rigidity for $\MP$-systems, a problem studied before. This gives several new rigidity results (up to some gauge) for this kind of Lorentzian manifolds.
\end{abstract}

\maketitle

\section{Introduction} 

Scattering rigidity is one of the main topics in geometric inverse problems. It has been studied in depth in the Riemannian setting \cites{pu2005,su, croke14, suv16, guillarmou17, suv21}. See the recent book \cite{psu} for an extensive treatment of this topic. Rigidity has been generalized to the magnetic case, in which one deals with magnetic geodesics \cites{dpsu, herrerossca, herrerosmag, zhou18}. A further generalization is the case of $\MP$-systems, in which, in addition to a magnetic part, the action of the potential is also involved. Rigidity for these system has been studied in \cites{jo2007, az} and by the author in \cites{mt23, mt24}. They also appear in other inverse problems, see \cites{iw, lz}. One could also deal with more general curves, as is the case of the works \cites{fsu, uv, ad, zhang}.

On the other hand, rigidity results are not so extense in the semi-Riemannian setting. A study of the scattering relation was done in \cites{losu, sy, plamen}. Light ray transforms have been studied in the Lorentzian setting as well \cites{stefanov89, stefanov17, wang18, losu20, fiko, fio, vw}. Regarding rigidity, there are some results. In \cite{hu} the authors consider the problem of reconstructing the topological, differentiable, and conformal structure of subsets $S \subset M^{\text{int}}$ by boundary observations of light cones emanating from points in $S$, with light rays being reflected at $\partial M$. In \cite{klu} it is shown that the family of light observation sets corresponding to point sources at points of a subset of a manifold, determine uniquely the conformal type this subset. In \cite{uyz}, boundary rigidity is studied for cylindrical domains in $\R^{1+n}$ endowed with standard stationary metrics, using time separation functions. In \cite{eskin}, the author studies rigidity of cylindrical domains in $\R^{1+n}$ endowed with time independent Lorentzian metrics using null geodesics. Finally, in \cite{plamen}, P.~Stefanov studied scattering rigidity of standard stationary manifolds using lightlike geodesics.

As in some of the previous works, in this article we will deal with standard stationary manifolds (see Section \ref{sec:geo_prelim} for the definition). The geometry of these manifolds have been studied extensively in \cites{gm, harris1, masiello, rs, gp, sanchez, fs, germinario, js, cfs, cs, bg, harris2, ahr, hj, sz}. A concrete example is the Kerr spacetime, which describes the geometry of spacetime around a spherically symmetric rotating body, see \cite{kerr} for an extensive treatment of this spacetime. These spaces also appear in acoustics, see \cite{visser}.

In this work we focus on scattering rigidity for standard stationary manifolds. Concretely, we address the following question: \emph{to what extent are standard stationary manifolds are ``timelike'' scattering rigid?}. Specifically, we define a scattering relation using timelike vectors only. Furthermore, we fix the momentum between these vectors and the vector field $\partial_{t}$. We obtain the following:

\begin{thm} \label{thm:intro1}
    Consider two metrics $g=g_{h,\omega,\lambda}$ and $g'=g_{h',\omega',\lambda'}$ on $M=\R \times N$, so that $h|_{\partial M}=h'|_{\partial M}$, $i^* \omega=i^* \omega'$, $\lambda|_{\partial M}=\lambda'|_{\partial M}$, where $i$ is the embedding $i \colon \partial M \to M$. Assume that the $\MP$-systems $(N,h,d\rho \omega,\rho^{2}/(-2\lambda))$ and $(N,h',\rho d\omega',\rho^{2}/(-2\lambda'))$ are simple, and one of the following conditions holds:
    \begin{enumerate}
    \item $h'$ is conformal to $h$;
    \item The $\MP$-systems are real-analytic;
    \item $\dim N=2$.
    \end{enumerate}
    If $\mathcal{S}_{\rho,m}=\mathcal{S}_{\rho,m}'$, then there exist a diffeomorphism $f \colon N \to N$ fixing $\partial N$, a positive function $\mu \in C^{\infty}(N)$ with $\mu|_{\partial M}=1$, and a function vanishing on the boundary $\varphi \in C^{\infty}(N)$ so that
    \begin{equation} \label{eq:ssm_mrho_equiv_intro}
        h'=\frac{1}{\mu}f^{*}h, \quad \omega'=f^{*}\omega+d\left( \frac{\varphi}{\rho} \right), \quad \frac{1}{\lambda'}=\mu \left( \frac{1}{f^{*}\lambda}-\frac{m^{2}}{\rho^{2}} \right) +\frac{m^{2}}{\rho^{2}}.
    \end{equation}
    Furthermore, if we assume that the previous conditions hold for two different values of $\rho$, then $\mu=1$ in \eqref{eq:ssm_mrho_equiv_intro}.
\end{thm}

\begin{thm} \label{thm:intro2}
There exists $m$ big enough such that for every $(h_0, \alpha_0, U_{0}) \in \mathcal{G}^m$ (introduced in Theorem \ref{thm:MP_gen}), there is $\eps>0$ such that for any two metrics $g=g_{h,\omega,\lambda}$ and $g'=g_{h',\omega',\lambda'}$ on $M=\R \times N$ with
    \begin{align*}
        \|h-h_0\|_{C^m(N)}+\|\omega-\omega_{0}\|_{C^m(N)}+\|\lambda-\lambda_{0}\|_{C^{m}(N)} &\leq \eps, \\
        \|h'-h_0\|_{C^m(N)}+\|\omega'-\omega_{0}\|_{C^m(N)}+\|\lambda'-\lambda_{0}\|_{C^{m}(N)} &\leq \eps,
    \end{align*}
    and $h|_{\partial M}=h'|_{\partial M}$, $i^* \omega=i^* \omega'$, $\lambda|_{\partial M}=\lambda'|_{\partial M}$, where $i$ is the embedding $i \colon \partial M \to M$, we have the following: if $\mathcal{S}_{\rho,m}=\mathcal{S}_{\rho,m}'$, then $g$ and $g'$ are related as in \eqref{eq:ssm_mrho_equiv_intro}. Furthermore, if the same hypothesis hold for two values of $\rho$, then the same conclusions as Theorem \ref{thm:intro1} are valid, i.e., \eqref{eq:ssm_mrho_equiv_intro} with $\mu=1$.
\end{thm}

We also show that one cannot obtain a strong relation than \eqref{eq:ssm_mrho_equiv_intro} with $\mu=1$ even if one assumes $\mathcal{S}_{\rho,m}=\mathcal{S}_{\rho,m}'$ for all $\rho$ (see Remark \ref{rmk:sharp}).

To prove the results, we relate the manifold $(M,g)$ to $\MP$-system by using the theory of reduction of Hamiltonian systems with symmetry. Then, we obtain a relation between the scattering relation of the manifold and the scattering relation of the $\MP$-system obtained in the previous step. Finally, we apply results from \cites{az, mt23, mt24}.

Let us finish the introduction by mentioning the organization of the paper. In Section \ref{sec:geo_prelim} we fix notation, and give definitions of standard stationary manifolds (SSM) and of the scattering relation. In Section \ref{sec:ham_red} we proceed with the Hamiltonian reduction and we relate the SSM with $\MP$-systems. In Section \ref{sec:rel} we show how the geometry of SSM and $\MP$-systems are related. This geometric relations will appear as technical conditions in the hypothesis of the main theorems, in Section \ref{sec:thms}, where we prove Theorems \ref{thm:intro1}, \ref{thm:intro2}. Finally, we have the appendices. In the first one, we gave some facts about rigidity of $\MP$-system. In Appendix \ref{app:ham_red}, we briefly review the results that we use for the Hamiltonian reduction. Finally, in Appendix \ref{app:light_geo}, we show the relation between the geometry of magnetic systems and SSM, in order to give a geometric meaning to some hypothesis in \cite{plamen}.
\\
\textbf{Acknowledgments.} The author would like to thank Plamen Stefanov for suggesting the problem, for helpful discussions and comments on a previous version of this manuscript. The author thanks Gabriel Paternain too, who suggested the symmetry point of view. Also, the author would like to thank Universidad de Chile for its hospitality during his visit. The author was partly supported by NSF Grant DMS-2154489.


\section{Geometric preliminaries} \label{sec:geo_prelim}

A connected Lorentzian manifold $(M, g)$ is called \emph{stationary spacetime} if it admits a timelike Killing vector field. In this note we will work with \emph{standard stationary spacetime}, that is, manifolds of the form $M=\R \times N$ with the following metric
\begin{equation} \label{eq:metric1}
    g_{ij}=-\lambda(x) dt^2+2 \tilde{\omega}_{i}(x) dt dx^{i}+\tilde{h}_{ij}(x)dx^{i}dx^{j},
\end{equation}
or, in matrix form, 
\begin{equation} \label{eq:metric_matrix}
    g=\begin{pmatrix}
    -\lambda & \tilde{\omega}_{1} & \cdots & \tilde{\omega}_{n} \\
    \tilde{\omega}_{1} & \tilde{h}_{11} & \cdots & \tilde{h}_{1n} \\
    \vdots & \vdots & \ddots & \vdots \\
    \tilde{\omega}_{n} & \tilde{h}_{n1} & \ddots & \tilde{h}_{nn}
\end{pmatrix}.
\end{equation}
where $\lambda \in C^{\infty}(N)$ is strictly positive, $\tilde{\omega}$ is a 1-form on $N$, and $\tilde{h}$ a Riemannian metric on $N$. We will always consider $N$ to be a smooth compact manifold with boundary.

Every stationary manifold is locally standard stationary. See \cite{cfs}*{Theorem 2.3} for a criterion that characterizes when a globally hyperbolic stationary space time is standard. To study geodesics on $(M,g)$,  we define the following metric on $N$
\[ h_{ij}dx^{i}dx^{j}=\tilde{h}_{ij}dx^{i}dx^{j}+\frac{1}{\lambda(x)} \tilde{\omega}_{i}\tilde{\omega}_{j}dx^{i}dx^{j}, \]
and let
\[ \omega=\frac{1}{\lambda}\tilde{\omega}. \]
Hence, we can write
\begin{equation} \label{eq:metric2}
    g_{ij}=-\lambda(x)(dt-\omega_{i} dx^{i})^{2} +h_{ij}(x)dx^{i}dx^{j},
\end{equation}
and
\begin{equation} \label{eq:metric_mat}
    \begin{pmatrix}
    -\lambda & \lambda \omega_{1} & \cdots & \lambda \omega_{n} \\
    \lambda \omega_{1} & -\lambda \omega_{1} \omega_{1}+h_{11} & \cdots & -\lambda\omega_{1} \omega_{n}+h_{1n} \\ 
    \vdots & \vdots & \ddots & \vdots \\
    \lambda \omega_{n} & -\lambda \omega_{n} \omega_{1}+h_{n1} & \cdots & -\lambda\omega_{n} \omega_{n}+h_{nn} \\ 
\end{pmatrix}.
\end{equation}
For $g$ as in \eqref{eq:metric2}, we will write $g=g_{\lambda,\omega,h}$. Note that since $g$ does not depend on time, $\partial_{t}$ is a Killing field. Hence, for any geodesic $\gamma=(t,x)$ we have that $(\dot{\gamma},\partial_{t})$ is constant along $\gamma$ (see \cite{oniell}*{p. 252}). Explicitly, 
\begin{equation} \label{eq:const2}
    -\lambda(x)(\dot{t}-\langle \omega,\dot{x} \rangle)=\rho,
\end{equation}
for some $\rho \in \R$. 

Geodesics $\gamma=(t,x)$ of $(M,g)$ are, by definition, critical points of the energy functional
\[ \int_{0}^{T}
(-\lambda(x)(\dot{t}-\langle \omega,\dot{x}\rangle)^{2}+|\dot{x}|_{h}^{2})ds. \]
For $\rho \in \R$, let $Y \colon TN \to TN$ the Lorentz force induced by $\rho d\omega$. By a variational argument, A. Germinario proved the following (cf. \cite{masiello}*{equation (3.4)}, \cite{fs}*{Theorem 2}, \cite{plamen}*{page 13}):

\begin{thm}[\cite{germinario}*{Theorem 3}] \label{thm:geo_sta_mp}
    If $\gamma=(t,x)$ is a geodesic for $(M,g)$ if and only if $t$ satisfies \eqref{eq:const2} for some $\rho \in \R$ and $x$ solves
    \begin{equation} \label{eq:mp_sta}
        \nabla_{\dot{x}}\dot{x}=Y(\dot{x})-\nabla \left(\frac{\rho^{2}}{-2\lambda}\right),
    \end{equation}
    where $\nabla$ denotes the Levi-Civita connection of the metric $h$. Reciprocally, geodesics $\gamma=(t,x)$ for $g$ can be obtained as solutions of \eqref{eq:mp_sta} for some $\rho \in \R$ and $t$ satisfying \eqref{eq:const2}. 
\end{thm}

In Section \ref{sec:ham_red} we will give a new proof of this result which will also give a relation between $(M,g)$ and $\MP$-systems. 

Now we give the definition of the scattering relation. We will consider timelike geodesics with $|\dot{\gamma}|_{g}^{2}=-m^2$. In physics, this means that the particle described by the geodesic has mass $m$, see \cite{sw}*{Section 3.1}. 

Note that by \cite{rs}*{Proposition 2.1}, $(M, g)$ is geodesically complete (see also \cite{bcfs}*{Theorem 1.1}, \cite{cs}*{Theorem 4.27}). Since we will work locally, we assume that we have two ``small" timelike surfaces $U$ and $V$ corresponding to initial and endpoints, respectively. We fix $p=(t,x) \in U, q=(s,y) \in V$ so that they are connected by a timelike geodesic $[0,T] \ni s \to \gamma(s)$, with $|\dot{\gamma}|_{g}^{2}=-m^2$. Assume that
\begin{equation} \label{eq:nc} \tag{NC}
    p \text { and } q \text { are not conjugate along } \gamma
\end{equation}
and that $\gamma$ is transversal to $U$ and $V$ at their only points of intersection, $p$ and $q$. Condition \eqref{eq:nc} is satisfied globally if the sectional curvature of $(M,g)$ is non-negative on all timelike tangent planes, see \cite{oniell}*{p. 277}. 

It is known that standard stationary manifolds are time-orientable (see also \cite{oniell}*{p. 145}). We fix a time orientation on $U$ that we call future pointing. Assume that $\gamma$ is future pointing at $p$, and we choose a time orientation on $V$ so that $\gamma$ is future pointing at $q$ as well. We also fix orientation on $U$ and $V$ in the classical sense, calling the sides containing $\gamma$ interior, and the other ones exterior. Set $v=\dot{\gamma}(0)$ and $w=\dot{\gamma}(1)$. Let $v'$, $w'$ be their orthogonal projections on $T U$ and $T V$, respectively; see \cite{oniell}*{p. 50}. They must be timelike. Let $\mathcal{U}, \mathcal{V}$ be small timelike conic neighborhoods of $(p, v')$ in $T U$, and of $(q,w')$ in $T V$, respectively. We define the scattering relation that we will use in this work as follows.

\begin{defin} \label{defin:rho_sca}
    The \emph{scattering relation of momentum $\rho$ and mass $m$} $\mathcal{S}_{\rho,m} \colon \mathcal{U} \to \mathcal{V}$ given by $\mathcal{S}_{\rho,m}(p, v')=(q, w')$ is defined as follows. Let $v$ be the timelike vector of mass $m$ at $p$ with orthogonal projection $v'$ on $T_p U$, and pointing to the interior; then $q \in V$ is the point where the geodesic $\gamma_{p,v}$ with $(\partial_{t},\dot{\gamma}_{p,v})_{g}=\rho$, issued from $(p, v)$ meets $V$, and $w'$ is the orthogonal projection on $T_q V$ of its direction there.
\end{defin}

\begin{rmk} \hfill
\begin{enumerate}
    \item We recall that we work with the projections in order to obtain a scattering relation invariant by diffeomorphism fixing the boundary. Otherwise, we would have to impose conditions in order to preserve the (inner) normal vectors at the boundary, see also Remark \ref{rmk:sca_proj}.
    \item We use the word \emph{momentum} in Definition \ref{defin:rho_sca} because we are fixing a momentum map, see Section \ref{sec:ham_red} and Appendix \ref{app:ham_red}.
    \item There is a relation between the angle of a vector (in the same timecone as $\partial_{t}$) and the momentum $\rho$, see Section \ref{sec:ener_ang} for details.
\end{enumerate}    
\end{rmk}

\section{Hamiltonian reduction} \label{sec:ham_red}

Consider the symplectic manifold $(TM,\eta)$, where $\eta$ is the pullback to $TM$ of the canonical symplectic form on $T^{*}M$ by the Lorentzian metric $g$. Write $G=(\R,+)$, and consider the following action of $G$ as a Lie group into $M$:
\begin{align*}
    \phi:G \times M & \to M, \\
    (s,(t,x)) & \mapsto(s+t,x).
\end{align*}
We lift the action of $G$ by $\phi$ to the tangent bundle of $M$, that is,
\begin{align*}
    \Phi \colon G \times TM & \to TM, \\
    (s,(t,x,v)) & \mapsto T\phi_{s}(t,x,v):=(\phi_{s}(t,x),(d\phi_{s})_{t,x}(v)).
\end{align*}
Since $\phi_{s}$ acts by isometries, in light of \cite{am}*{Corollary 3.2.14} we conclude that $\Phi$ is a symplectic action on $(TM,\eta)$. 

Consider the Lie algebra of $G$
\[ \mathfrak{g}=T_{0}G=\R \frac{\partial}{\partial t}\bigg|_{t=0}, \]
and its dual, given by
\[ \mathfrak{g}^{*}=T_{0}^{*}G=\R dt|_{t=0}. \]

Let $\xi=\xi \frac{\partial}{\partial t}|_{t=0} \in \mathfrak{g}$. Then, $s\xi$ is the integral curve of $\xi$ that passes through $0$ at time $0$. $\xi$ and $\Phi$ induce an $\R$-action on $TM$ on the following way:
\begin{align*}
    \R \times TM &\to TM, \\
    (s,(t,x,v)) & \mapsto \Phi(s\xi,(t,x,v)).
\end{align*}
The vector field on $TM$ given by 
\[ \xi_{TM}(t,x,v)=\frac{d}{ds}\bigg|_{s=0}\Phi(s\xi,(t,x,v))=\xi \frac{\partial}{\partial t}\bigg|_{t=0}, \]
is the \emph{infinitesimal generator} of the action corresponding to $\xi$ \cite{am}*{Definition 4.1.24}.

Now we are ready to find a momentum map (see Definition \ref{def:mo}). 

\begin{lemma}
An $Ad^{*}$-equivariant momentum map for the action $\Phi$ of $G$ in $(TN,\eta)$ is 
\[J(t,x,v)=-\lambda(x)(v_{t}-\langle \omega(x),v_{x} \rangle)dt|_{t=0}.\]
\end{lemma}

\begin{proof}
We can apply Lemma \ref{lemma:mo_tan}, to obtain
\[ \hat{J}(\xi)(v)=((v_{t},v_{x}), \xi_Q(q))_{g}=-\xi\lambda(x)(v_{t}-\langle \omega(x),v_{x} \rangle). \]
Then, by definition (see Definition \ref{def:mo}), we find 
\[ J(t,x,v)=-\lambda(x)(v_{t}-\langle \omega(x),v_{x} \rangle)dt|_{t=0}. \]
\end{proof}

\begin{rmk} \label{rmk:timecone}
Note that if $(x,v) \in J^{-1}(\rho)$ and $v$ is timelike, we have the following three cases
\begin{itemize}
    \item If $\rho<0$, then $v$ and $\partial_{t}$ are in the same timecone and
    \begin{equation} \label{eq:v_t}
        v_{t}=-\frac{\rho}{\lambda}+\langle \omega,v_{x} \rangle.
    \end{equation}
    Furthermore, there is a unique $\varphi \geq 0$ (the \emph{hyperbolic angle}, see \cite{oniell}*{p. 144}) with
    \[ \rho=-\sqrt{\lambda} |v|_{g}\cosh\varphi. \]
    \item If $\rho=0$, then $v$ and $\partial_{t}$ are orthogonal and $v_{t}=\langle \omega,v_{x} \rangle$;
    \item If $\rho>0$, then $-v$ and $\partial_{t}$ are in the same timecone and \eqref{eq:v_t}. As before, there exists a unique $\varphi \geq 0$ with
    \[ \rho=\sqrt{\lambda} |v|_{g} \cosh \varphi. \]
\end{itemize}
\end{rmk}

The tuple $(TM,\eta,\Phi,J)$ is a \emph{Hamiltonian $G$-space} (see Definition \ref{def:mo}). We also consider the Hamiltonian
\begin{align*}
    H \colon TM &\to \R, \\
    (t,x,v) & \mapsto \frac{1}{2}|v|_{g(t,x)}^{2}.
\end{align*}

\begin{lemma} \label{lemma:sym_red}
The symplectic reduction of the Hamiltonian $G$-space $(TM,\eta,\Phi,J)$ endowed with the Hamiltonian $H$ is given by the manifold $TN$ 
endowed with the symplectic form 
\begin{equation} \label{eq:eta_rho}
    \eta_{\rho}=\delta^{\flat}+\pi^{*}(\rho d\omega),
\end{equation}
where $\delta^{\flat}$ is the pullback of the canonical 2-form on $T^{*}N$ $\delta$, to $TN$ via $h$, and $\pi \colon TN \to N$ is the base point projection. The reduced Hamiltonian is given by
\begin{equation} \label{eq:H_rho}
    H_{\rho}(x,v)=\frac{1}{2}|v_{x}|_{h(x)}^{2}+\frac{\rho^{2}}{-2\lambda(x)}.
\end{equation}
\end{lemma}

\begin{proof}
The idea is to apply Theorem \ref{thm:mwm}.
\begin{itemize}
    \item \textbf{Step 1:}  The setting is the following. We consider the Hamiltonian $G$-space $(TM,\eta,\Phi,J)$. Recall that the action is symplectic, and $J$ is $Ad^{*}$-equivariant. Note that one of the entries of the differential of $J$ is $-\lambda$. By hypothesis, $\lambda>0$. Hence, all the points on $TM$ are regular points of $J$. Take $\rho=\rho dt|_{t=0}$ to be one of them. Then,
    \[ J^{-1}(\rho)=\{(t,x,v) \in TM:-\lambda(x)(v_{t}-\langle \omega(x),v_{x} \rangle)=\rho\}. \]
    is a submanifold of $TM$ of codimension $1$. Let $G_{\rho}$ denote the isotropy group of $G$ under the co-adjoint action $Ad^{*}$, that is,
    \[ G_{\rho}=\{g \in G:Ad_{g^{-1}}^{*}\rho=\rho\}. \]
    Observe that $G$ is commutative, then $Ad_{g}^{*}=id_{\mathfrak{g}^{*}}$. Therefore, $G_{\rho}=G=\R$.
    \item \textbf{Step 2:} Now, we need to show that $G$ acts freely and properly on $J^{-1}(\rho)$. 
    
    To see that $G$ acts freely, we need to check that for each $(t,x,v) \in TM$, $s \mapsto \Phi_{s}(t,x,v)$ is injective. This is clear. Indeed, looking at the first entry we find that if $\Phi_{s}(t,x,v)=\Phi_{s'}(t,x,v)$, then $s=s'$.
    
    To see that $G$ acts properly on $TM$, we need to prove that 
    \begin{align*}
        \Tilde{\Phi} \colon G \times J^{-1}(\rho) & \to J^{-1}(\rho) \times J^{-1}(\rho),\\
        (s,(t,x,v)) & \mapsto ((t,x,v),\Phi(s,(t,x,v))),
    \end{align*}
    is a proper map. Let $K=K_{1} \times K_{2} \subset J^{-1}(\rho) \times J^{-1}(\rho)$ be a compact set. Since $\Phi$ only affects the time component, we see that 
        \[ \Tilde{\Phi}^{-1}(K)=\{s \in G, \,(t,x,v) \in J^{-1}(\rho):(t,x,v) \in K_{1},\Phi_{s}(t,x,v) \in K_{2}\}, \]
    is compact. 
    \item \textbf{Step 3:} By the previous steps, we can invoke Theorem \ref{thm:mwm} to obtain that the reduced manifold is
    \begin{align*}
        M_{\rho}&=J^{-1}(\rho)/G \\
        &=\left\{ \left( t,x, -\frac{\rho}{\lambda}+\langle \omega,v_{x}\rangle, v_{x} \right):(x,v) \in TM \right\}/\R \\
        &=TN,
    \end{align*}
    and is endowed with the symplectic form $\eta_{\rho}$ satisfying 
    \[ \pi_{\rho}^{*}\eta_{\rho}=i_{\rho}^{*}\eta. \]
    where $\pi_\rho \colon J^{-1}(\rho) \to TN$ is the canonical projection and $i_\rho: J^{-1}(\rho) \to TM$ is the inclusion. To find $\eta_{\rho}$, first we observe that in local coordinates, $\eta$ is given by
    \[\eta=\sum_{i,j=0}^{n} \left \lbrace  \left(\sum_{k=0}^{n} v^{k} \partial_{x^{j}}g_{ik}  \right) dx^{i} \wedge dx^{j}+g_{ij}dx^{i}\wedge dv^{j}  \right \rbrace.\]
    Here $x^{0}=t$ and $v^{0}=v_{t}$. Using \eqref{eq:metric_mat}, we see that 
    \begin{align*}
        \eta=&-\lambda dt \wedge dv^{0}-\sum_{j=1}^{n}v^{0}\partial_{x^{j}} \lambda dt \wedge dx^{j}+\sum_{j,k=1}^{n}v_{x}^{k}\partial_{x^{j}}(\lambda \omega_{k}) dt\wedge dx^{j} \\
        &+\sum_{j=1}^{n} \lambda \omega_{j}dt \wedge dv_{x}^{j}+ \sum_{i=1}^{n} \lambda \omega_{i}dx^{i} \wedge dv^{0}+\sum_{i,j=1}^{n}v^{0}\partial_{x^{j}}(\lambda \omega_{i})dx^{i}\wedge dx^{j} \\
        &-\sum_{i,j,k=1}^{n} v_{x}^{k}\partial_{x^{j}}(\lambda \omega_{i} \omega_{k})dx^{i} \wedge dx^{j}+\sum_{i,j,k=1}^{n}v_{x}^{k}\partial_{x^{j}}h_{ik}dx^{i} \wedge dx^{j} \\
        &- \sum_{i,j=1}^{n} \lambda \omega_{i}\omega_{j} dx^{i} \wedge dv_{x}^{j} +\sum_{i,j=1}^{n}h_{ij}dx^{i} \wedge dv_{x}^{j}.
    \end{align*}
    Using this and that $J \equiv \mu$, is nor hard to see then that $\eta_{\rho}$ is as in \eqref{eq:eta_rho}. Finally, since $G$ acts by isometries, we have that $H$ is invariant by the action of $G$. Then, by Theorem \ref{thm:ham_red}, there is a reduced Hamiltonian $H_{\rho}$ on $M_{\rho}$ satisfying
    \[ H_{\rho} \circ \pi_{\rho}=H \circ i_{\rho}. \]
    Using $J \equiv \mu$, we obtain that $H_{\rho}$ has the form \eqref{eq:H_rho}.
\end{itemize}
\end{proof}

Now we give the relation between stationary manifolds and $\MP$-systems, which generalizes and gives a new proof of \cite{germinario}*{Theorem 3}.

\begin{thm} \label{thm:rel_lorentz_mp}
    Let $N$ be a smooth manifold with boundary, and write $M=\R \times N$. Consider the standard stationary Lorentz manifold $(M,g)$ so that $g$ is as in \eqref{eq:metric2}. 
    \begin{enumerate}
        \item The projection of every trajectory of the geodesic flow on $(TM,g)$ corresponds to a trajectory of the Hamiltonian flow on $(M_{\rho},\eta_{\rho},H_{\rho})$ for some $\rho \in \R$. Reciprocally, given a trajectory $(x,v_{x})$ of the Hamiltonian flow on $(TN,\eta_{\rho},H_{\rho})$, one can find a trajectory of the flow on $(TM,\eta,H)$ whose projection to $TN$ is $(x,v_{x})$.
        \item The projection of every geodesic in $(M,g)$ corresponds an $\MP$-geodesic on the $\MP$-system $(N,h,d\rho \omega,\rho^{2}/(-2\lambda))$, where $\rho \in \R$. Reciprocally, given an $\MP$-geodesic $x$ of the $\MP$-system $(N,h,d\rho \omega,\rho^{2}/(-2\lambda))$, one can find a geodesic on $(M,g)$ whose projection to $N$ is $x$. In particular, timelike geodesic with $|\gamma|_{g}^{2}=-m^2$ corresponds to $\MP$-geodesics of energy $-m^{2}/2$. Furthermore, there are not timelike geodesics orthogonal to $\partial_{t}$.
    \end{enumerate}
\end{thm}

\begin{proof} \hfill
    \begin{enumerate}
    \item The first part follows Theorem \ref{thm:ham_red} and Lemma \ref{lemma:sym_red}. For the second part, we will find the Hamilton's equations for the Hamiltonian system $(TM,\eta, H)$. To this end, we examine the equation
    \begin{equation} \label{eq:flow}
        \eta(X_{H},V)=dH(V).
    \end{equation}
    We write in coordinates $X_{H}=A^{0}\partial_{t}+A^{i}\partial_{x^{i}}+B^{0}\partial_{v^{0}}+B^{i}\partial_{v_{x}^{i}}$ and $V=C^{0}\partial_{t}+C^{i}\partial_{x^{i}}+D^{0}\partial_{v^{0}}+D^{i}\partial_{v_{x}^{i}}$.Then, the right-hand side of \eqref{eq:flow} is
    \begin{align*}
        dH(V)=&\left( -\frac{1}{2}\partial_{x^{k}}\lambda (v^{0})^{2}+\partial_{x^{k}}(\lambda \omega_{i})v^{0}v_{x}^{i} \right. \\
        &\left.-\frac{1}{2}\partial_{x^{k}}(\lambda \omega_{i}\omega_{j})v_{x}^{i}v_{x}^{j}+\frac{1}{2}\partial_{x^{k}}h_{ij}v_{x}^{i}v_{x}^{j} \right) C^{k} \\
        &+(-\lambda v^{0}+\lambda \omega_{i}v_{x}^{i})D^{0} \\
        &+(\lambda v^{0}\omega_{k}-\lambda \omega_{i}\omega_{k}v_{x}^{i}+h_{ik}v_{x}^{i})D^{k},
    \end{align*}
    while the left-hand side of \eqref{eq:flow} is
    \begin{align*}
        \eta(X_{H},V)=&-\lambda(A^{0}D^{0}-B^{0}C^{0})-v^{0}\partial_{x^{j}}\lambda (A^{0}C^{j}-C^{0}A^{j}) \\
        &+v_{x}^{k}\partial_{x^{j}}(\lambda \omega_{k})(A^{0}C^{j}-C^{0}A^{j})+\lambda \omega_{j}(A^{0}D^{j}-C^{0}B^{j}) \\
        &+\lambda \omega_{i}(A^{i}D^{0}-B^{0}C^{i})+v^{0}\partial_{x^{j}}(\lambda \omega_{i})(A^{i}C^{j}-C^{i}A^{j}) \\
        &-v_{x}^{k}\partial_{x^{j}}(\lambda \omega_{i}\omega_{k})(A^{i}C^{j}-C^{i}A^{j})+v_{x}^{k}\partial_{x^{j}}h_{ik}(A^{i}C^{j}-C^{i}A^{j}) \\
        &-\lambda \omega_{i}\omega_{j}(A^{i}D^{j}-C^{i}B^{j})+h_{ij}(A^{i}D^{j}-C^{i}B^{j}).
    \end{align*}
    Let $C^{0}=0$, $C^{k}=0$ and $D^{k}=0$, $D^{0} \neq 0$ in \eqref{eq:flow}, and use $\lambda>0$ to find
    \begin{equation} \label{eq:flow1}
        -v^{0}+\omega_{i}v_{x}^{i}=-A^{0}+\omega_{i} A^{i}.
    \end{equation}
    Now let $C^{0}=0$, $C^{k}=0$, $D^{0}=0$, and $D^{k} \neq 0$ in \eqref{eq:flow}. We get
    \begin{equation} \label{eq:flow2}
        \lambda v^{0}\omega_{k}-\lambda \omega_{i}\omega_{k}v_{x}^{i}+h_{ik}v_{x}^{i}=\lambda \omega_{k}A^{0}-\lambda \omega_{i}\omega_{k}A^{i}D^{k}+h_{ik}A^{i}D^{k}.
    \end{equation}
    Using \eqref{eq:flow1} and \eqref{eq:flow2} we see that $A^{0}=v^{0}$ and $A^{i}=v_{x}^{i}$. Now, let $C^{k}=0$, $D^{0}=0$, $D^{k}=0$ and $C^{0} \neq 0$ in \eqref{eq:flow} to obtain
    \begin{equation} \label{eq:flow3}
        \lambda B^{0}+v^{0}\partial_{x^{j}}\lambda v_{x}^{j}-v_{x}^{k}\partial_{x^{j}}(\lambda \omega_{k})v_{x}^{j}-\lambda \omega_{j}B^{j}=0.
    \end{equation}
    On the other hand, if we put $C^{0}=0$, $D^{0}=0$, $D^{k}=0$ and $C^{k} \neq 0$, it follows that
    \begin{align*}
        &-\frac{1}{2}\partial_{x^{k}}\lambda (v^{0})^{2}+\partial_{x^{k}}(\lambda \omega_{i})v^{0}v^{i}-\frac{1}{2}\partial_{x^{k}}(\lambda \omega_{i}\omega_{j})v_{x}^{i}v_{x}^{j}+\frac{1}{2}\partial_{x^{k}}h_{ij}v_{x}^{i}v_{x}^{j} \\
        =&-v^{0}\partial_{x^{k}}\lambda v^{0}+v_{x}^{j}\partial_{x^{k}}(\lambda \omega_{j})v^{0}-\lambda \omega_{k}B^{0}+v^{0}\partial_{x^{k}}(\lambda \omega_{i})v_{x}^{i}- v^{0}\partial_{x^{j}}(\lambda \omega_{k})v_{x}^{j} \\
        &-v_{x}^{j}\partial_{x^{k}}(\lambda \omega_{i}\omega_{j})v_{x}^{i}+v_{x}^{i}\partial_{x^{j}}(\lambda \omega_{k}\omega_{i})v_{x}^{j}+v_{x}^{j}\partial_{x^{k}}h_{ij}v_{x}^{i}- v_{x}^{i}\partial_{x^{j}}h_{ik}v_{x}^{j} \\
        &+\lambda \omega_{k}\omega_{j}B^{j}-h_{kj}B^{j}.
    \end{align*}
    Using this and \eqref{eq:flow3} we obtain
    \begin{equation} \label{eq:flow4}
        B^{\ell}=-\Gamma_{ij}^{\ell}v_{x}^{i}v_{x}^{j}-\frac{1}{2}(v^{0}-\omega_{j}v_{x}^{j})^{2}h^{k\ell}\partial_{x^{k}}\lambda+\lambda (v^{0}-\omega_{j}v_{x}^{j})h^{k\ell}(\partial_{x^{k}}\omega_{i}-\partial_{x^{i}}\omega_{k})v_{x}^{i}.
    \end{equation}
    Therefore, Hamilton's equations take the form
    \begin{align*}
        \dot{t}=&v^{0}, \\
        \dot{x}^{\ell}=&v_{x}^{\ell}, \\
        \dot{v}^{0}=&-\frac{1}{\lambda}v^{0}\partial_{x^{j}}\lambda v_{x}^{j}+\frac{1}{\lambda} \partial_{x^{j}}(\lambda \omega_{k})v_{x}^{j}v_{x}^{k}-\omega_{\ell}\Gamma_{ij}^{\ell}v_{x}^{i}v_{x}^{j} \\
        &-\frac{1}{2}(v^{0}-\omega_{j}v_{x}^{j})^{2}h^{k\ell}\omega_{\ell}\partial_{x^{k}}\lambda-\lambda(v^{0}-\omega_{j}v_{x}^{j})h^{k\ell}\omega_{\ell}(\partial_{x^{k}}\omega_{i}-\partial_{x^{i}}\omega_{k})v_{x}^{i}, \\
        \dot{v}_{x}^{\ell}=&-\Gamma_{ij}^{\ell}v_{x}^{i}v_{x}^{j}-\frac{1}{2}(v^{0}-\omega_{j}v_{x}^{j})^{2}h^{k\ell}\partial_{x^{k}}\lambda+\lambda (v^{0}-\omega_{j}v_{x}^{j})h^{k\ell}(\partial_{x^{k}}\omega_{i}-\partial_{x^{i}}\omega_{k})v_{x}^{i}.
    \end{align*}    
    Take a trajectory $(x,v_{x})$ of the Hamiltonian flow on $(TN,\eta_{\rho},H_{\rho})$. Then $(x,v_{x})$ satisfy \cite{mt23}*{Equations (A.1), (A.2)}. Define $(t,v^{0})$ to satisfy
    \begin{align*}
        \dot{t}=v^{0}, \\
        -\lambda(v^{0}-\omega_{j}v_{x}^{j})=\rho.
    \end{align*}
    Then, $(t,x,v^{0},v_{x})$ satisfies the Hamilton's equations for the system $(TM,\eta,H)$ described before. 
    \item Take a geodesic $\gamma=(t,x)$ on $(M,g)$. Then $(\gamma,\dot{\gamma})$ satisfy Hamilton's equations of the Hamiltonian system $(TN,\eta,H)$, because these equations are exactly the geodesics equations. By the previous part of the theorem, the projection $(x,\dot{x})$ satisfy Hamilton's equations of the Hamiltonian system $(TN,\eta_{\rho},H_{\rho})$. However, these last Hamiltonian's equations are exactly the equations of $\MP$-geodesics, see \cite{mt23}*{Equations (A.1), (A.2)}. This proves the first part.

    For the second part, consider $x$ to be a $\MP$-geodesic of the $\MP$-system $(N,h,d\rho \omega,\rho^{2}/(-2\lambda))$. Define $t$ to satisfy \eqref{eq:const2}. Then, $\gamma:=(t,x)$ is a geodesic on the manifold $(M,g)$.
    
    Now take a geodesic with $|\dot{\gamma}|_{g}^{2}=-m^2$. This can be written as 
    \begin{equation} \label{eq:timelike}
    -\lambda(x)(\dot{t}-\langle \omega,\dot{x}\rangle)^{2}+|\dot{x}|_{h}^{2}=-m^{2},
    \end{equation}
    where $\gamma=(t,x)$. This together with \eqref{eq:const2} gives
    \begin{equation} \label{eq:2e}
        \frac{\rho^{2}}{-\lambda}+|\dot{x}|_{h}^{2}=-m^{2}.
    \end{equation}
    Then, the $\MP$-geodesic $x$, associated to the geodesic $\gamma$ has constant energy $-\frac{m^{2}}{2}$.

    For the last part, assume by contradiction that there exist a geodesic $\gamma=(t,x)$ with $|\dot{\gamma}|_{g}^{2}=-m^{2}$ so that $(\dot{\gamma},\partial_{t})_{g}=0$. Then 
    \begin{equation} \label{eq:J=0}
        \lambda (\dot{t}-\langle \omega, \dot{x} \rangle)=0.
    \end{equation}
    This together with \eqref{eq:timelike} imply that $|\dot{x}|_{h}^{2}=-m^{2}$. This contradiction finishes the proof.
\end{enumerate}
\end{proof}

\section{Relations between Standard Stationary Manifolds and MP-systems} \label{sec:rel}

Since we want to use rigidity results for simple $\MP$-systems, in this section we explore the connections between the symplicity of the system $(N,h,d\rho\omega,\rho^{2}/(-2\lambda))$ and the geometry of $(M,g)$. Notions about $\MP$-systems are given in Appendix \ref{app:mp}.

\subsection{Geodesic conectedness and the exponential map} \label{sec:con_exp}

The following lemma shows the correspondence between the fact that the $\MP$-exponential map to be a diffeomorphism and the existence of timelike geodesics joining points with integral lines of $\partial_{t}$.

\begin{lemma}
    Let $p=(t_{1},y)$, $q=(t_{2},z)$ with $y \neq z$ and set $\ell=\{(s,z):s \in \R\}$. Then, there exists a unique timelike geodesic joining $p$ with $\ell$ so that $(\partial_{t},\dot{\gamma})_{g}=\rho$ and $|\dot{\gamma}|_{g}^{2}=-m^{2}$ if and only if the $\MP$-exponential map of energy $-m^{2}/2$ of the system $(N,h,d\rho \omega,\rho^{2}/(-2\lambda))$, $\exp_{y}^{-m^{2}/2}$, is a diffeomorphism.
\end{lemma}

\begin{proof}
Let $y \neq z$ be points in $N$. Let $\gamma=(t,x)$ be the unique timelike geodesic joining $p$ with $\ell$ with $(\partial_{t},\dot{\gamma})_{g}=\rho$ and $|\gamma|=-m^{2}$. Then, in virtue of Theorem \ref{thm:rel_lorentz_mp}, $x$ is a $\MP$-geodesic of the system $(N,h,d\rho\omega, \rho^{2}/(-2\lambda))$ with $x(0)=y$, $E(x,\dot{x})=-m^{2}/2$ and $x(s_{0})=z$ for some $s_{0}$. Then, $\exp_{y}^{-m^{2}/2}(s_{0}\dot{x}(0))=z$, which shows that $\exp_{y}^{-m^{2}}$ is surjective. If we have $\exp_{y}^{-m^{2}}(s_{1}v_{1})=\exp_{y}^{-m^{2}}(s_{2}v_{2})=z$, then there are two $\MP$-geodesics $x_{1},x_{2}$, both of energy $-m^{2}/2$ so that $x_{j}(0)=y$, $\dot{x}_{1}(0)=v_{1}$, $\dot{x}_{2}=v_{2}$ and $x_{1}(s_{1})=x_{2}(s_{2})=z$. Define $t$ by \eqref{eq:const2} and initial condition $t(0)=t_{1}$. Then, $\gamma_{j}:=(t,x_{j})$ are two timelike geodesics joining $z$ with $V$ so that $(\partial_{t},\dot{\gamma})_{g}=\rho$ and $|\dot{\gamma}|_{g}^{2}=-m^{2}$, which contradicts the uniqueness condition. This shows injectivity of the exponential map. Since the $\MP$-exponential map is a local diffeomorphism (the proof is the same as the usual one in the Riemannian case), we obtain the desired conclusion.

Reciprocally, let $p=(t_{1},y)$, $q=(t_{2},z)$ with $y \neq z$ be points in $M$. Since $\exp_{y}^{-m^{2}/2}$ is a diffeomorphism, there exist a unique $v_{x}$ of energy $-m^{2}/2$ and a unique $s_{0}$ so that $\exp_{y}^{-m^{2}/2}(s_{0}v_{x})=z$. This means that there exists a $\MP$-geodesic $x$ of energy $-m^{2}/2$ so that $x(0)=y$, $\dot{x}(0)=v_{x}$ and $x(s_{0})=z$. Define $t$ by \eqref{eq:const2} and initial condition $t(0)=t_{1}$. Therefore, $\gamma:=(t,x)$ is a timelike geodesic joining $z$ with $V$ so that $(\partial_{t},\dot{\gamma})_{g}=\rho$ and $|\dot{\gamma}|=-m^{2}$. If there exists another $\gamma$ satisfying the previous conditions, say $\gamma=(\tilde{t},\tilde{x})$, then $\tilde{x}$ would satisfies the same conditions as $x$, contradicting the injectivity of the exponential map.
\end{proof}

By \cite{masiello}*{Theorem 3.5.2}, on $M$ as before, we can always join a point with an integral line of $\partial_{t}$ by a timelike geodesic. However, since we are working with a fixed mass, the result is not enough to ensure the properties that we need about $\exp^{-m^{2}/2}$. Furthermore, the result does not ensure that we can choose a geodesic with $(\partial_{t},\dot{\gamma})_{g}=\rho$, that is, $\rho$ could change from geodesic to geodesic. A way to ensure that the exponential map of energy $-m^{2}/2$ of the $\MP$-system $(N,h,d\rho \omega,\rho^{2}/(-2\lambda))$ to be a diffeomorphism would be the existence of \emph{brachistochrones}, which are the curves in the set
\[ \mathscr{B}_{p,\ell}(\rho)=\{\gamma \in C^{1}([0,T_{\gamma}],M):\gamma(0)=p, \, \gamma(T_{\gamma})\in \ell, |\dot{\gamma}|_{g}^{2}=-1, \, (\partial_{t},\dot{\gamma})_{g}=\rho\}. \]
The existence of this type of curves is studied in \cite{gp}. In our case, the existence, assuming \eqref{eq:cond_lam}, would depends on whether or not $M$ is complete with the Riemannian metric
\[ g(v,v)=\lambda(v_{0}-\langle \omega,v_{x}\rangle)^{2}+|v_{x}|_{h}^{2}. \]

\subsection{Convexity} \label{sec:conv}

To undestand whtat $\MP$-convextiy of the system $(N,h,d\rho \omega,\rho^{2}/(-2\lambda))$ means for $(M,g)$, we will relate the second fundamental forms of them. To this end, we first compute the Christoffel's symbols of $g$. Although one could reed them from Hamiltonian equations, it is not clear from the way in which we write them. So, using Sherman--Morrison--Woodbury formula, we first find that the inverse of \eqref{eq:metric_mat} is given by
\[ \begin{pmatrix}
    \frac{1}{-\lambda}+h^{k\ell}\omega_{k}\omega_{\ell} & h^{kj}\omega_{k} \\ h^{i\ell}\omega_{\ell} & h^{ij}
\end{pmatrix}. \]
Then, we compute and obtain 
\begin{align*}
    ^g \Gamma_{00}^{0} =& \frac{1}{2}h^{sm}\omega_{s}\partial_{x^{m}}\lambda, \\
    ^g \Gamma_{i0}^{0} =& \frac{1}{2\lambda}\partial_{x^{i}}\lambda+\frac{1}{2}h^{sm}\omega_{s}(\partial_{x^{i}}(\lambda \omega_{m})-\partial_{x^{m}}(\lambda \omega_{i})), \\
    ^g \Gamma_{ij}^{0} =&-\frac{1}{2\lambda}(\partial_{x^{j}}(\lambda \omega_{i})+\partial_{x^{i}}(\lambda \omega_{j}))+\omega_{s}\Gamma_{ij}^{s}\\ &-\frac{1}{2}h^{sm}\omega_{s}(\partial_{x^{j}}(\lambda \omega_{m}\omega_{i})+\partial_{x^{i}}(\lambda \omega_{m}\omega_{j})-\partial_{x^{m}}(\lambda \omega_{i}\omega_{j})), \\
    ^g \Gamma_{00}^{\ell}=&\frac{1}{2}h^{\ell m}\partial_{x^{m}}\lambda, \\
    ^g \Gamma_{i0}^{\ell}=&\frac{1}{2}h^{\ell m}(\partial_{x^{i}}(\lambda \omega_{m})-\partial_{x^{m}}(\lambda \omega_{i}))-\frac{1}{2}h^{t\ell}\omega_{t}\partial_{x^{i}}\lambda, \\
    ^g \Gamma_{ij}^{\ell} =&\frac{1}{2}h^{\ell t}\omega_{t}(\partial_{x^{j}}(\lambda \omega_{i})+\partial_{x^{i}}(\lambda \omega_{j}))+^h \Gamma_{ij}^{\ell}\\ &-\frac{1}{2}h^{\ell m} (\partial_{x^{j}}(\lambda \omega_{m}\omega_{i})+\partial_{x^{i}}(\lambda \omega_{m} \omega_{j})-\partial_{x^{m}}(\lambda \omega_{i} \omega_{j}) ).
\end{align*}
We observe that $\nu=(\langle \omega,\nu_{x}\rangle ,\nu_{x})$ is normal to $\partial M$ and unit exterior, where $\nu_{x}$ is the exterior unit normal to $\partial N$. Then, 
\[ ^g \nabla_{v}\nu
    =v^{j}\partial_{x^{j}}\nu^{k}+\langle \omega,\nu_{x}\rangle v^{0} \Gamma_{00}^{k}+\langle \omega,\nu_{x}\rangle v^{j} \Gamma_{0j}^{k}+\nu^{i}v^{0}\Gamma_{i0}^{k}+\nu^{i}v^{j} \Gamma_{ij}^{k}. \]
So, 
\begin{align*}
    (^g \nabla_{v}\nu,v)_{g} =&-\lambda(v^{0}-\langle \omega,v_{x}\rangle)(v_{x}^{j}\partial_{x^{j}}\langle \omega,\nu_{x} \rangle+\langle \omega,\nu_{x} \rangle v^{0}\Gamma_{00}^{0}+\langle \omega,\nu_{x}\rangle v_{x}^{j}\Gamma_{0j}^{0}+\nu_{x}^{i}v^{0}\Gamma_{i0}^{0}+\nu_{x}^{i}v_{x}^{j}\Gamma_{ij}^{0} \\
    &-\omega_{\ell}(v_{x}^{j}\partial_{x^{j}}\nu_{x}^{\ell}+\langle \omega,\nu_{x}\rangle v^{0}\Gamma_{00}^{\ell}+\langle \omega,\nu_{x}\rangle v_{x}^{j} \Gamma_{0j}^{\ell}+\nu_{x}^{i}v^{0}\Gamma_{i0}^{\ell}+\nu_{x}^{i}v_{x}^{j} \Gamma_{ij}^{\ell})) \\
    &+h_{k\ell}v_{x}^{k}(v_{x}^{j}\partial_{x^{j}}\nu_{x}^{\ell}+\langle \omega,\nu_{x}\rangle v^{0}\Gamma_{00}^{\ell}+\langle \omega,\nu_{x}\rangle v_{x}^{j} \Gamma_{0j}^{\ell}+\nu_{x}^{i}v^{0}\Gamma_{i0}^{\ell}+\nu_{x}^{i}v_{x}^{j} \Gamma_{ij}^{\ell}) \\
    =&(^h \nabla_{v_{x}}\nu_{x},v_{x})_{h}-\lambda v^{0}v_{x}^{j}\omega_{k}\partial_{x^{j}}\nu_{x}^{k}-\lambda v^{0}v^{j}\nu_{x}^{k}\partial_{x^{j}}\omega_{k} \\
    &-\lambda (v^{0})^{2} \omega_{k}\nu_{x}^{k}\frac{1}{2}h^{sm}\omega_{s}\partial_{x^{m}}\lambda \\
    &-\lambda v^{0} \omega_{k}\nu_{x}^{k} v_{x}^{j}\frac{1}{2\lambda}\partial_{x^{j}}\lambda-\lambda v^{0} \omega_{k}\nu_{x}^{k} v_{x}^{j} \frac{1}{2}h^{sm}\omega_{s}(\partial_{x^{j}}(\lambda \omega_{m})-\partial_{x^{m}}(\lambda \omega_{j})) \\
    & -\lambda (v^{0})^{2} \nu_{x}^{i}\frac{1}{2\lambda}\partial_{x^{i}}\lambda-\lambda (v^{0})^{2} \nu_{x}^{i} \frac{1}{2}h^{sm}\omega_{s}(\partial_{x^{i}}(\lambda \omega_{m})-\partial_{x^{m}}(\lambda \omega_{i})) \\
    &+\lambda v^{0} \frac{1}{2\lambda}(\partial_{x^{j}}(\lambda \omega_{i})+\partial_{x^{i}}(\lambda \omega_{j}))\nu_{x}^{i}v_{x}^{j}-\lambda v^{0}\omega_{s}\Gamma_{ij}^{s}\nu_{x}^{i}v_{x}^{j} \\
    &+\lambda v^{0}\frac{1}{2}h^{sm}\omega_{s}(\partial_{x^{j}}(\lambda \omega_{m}\omega_{i})+\partial_{x^{i}}(\lambda \omega_{m}\omega_{j})-\partial_{x^{m}}(\lambda \omega_{i}\omega_{j}))\nu_{x}^{i}v_{x}^{j} \\
    &+\lambda v^{0} \omega_{\ell} v_{x}^{j}\partial_{x^{j}}\nu_{x}^{\ell}+\lambda (v^{0})^{2}\omega_{\ell}\omega_{k}\nu_{x}^{k}\frac{1}{2}h^{\ell m} \partial_{x^{m}} \lambda \\
    &+\lambda v^{0} \omega_{\ell} \omega_{k}\nu_{x}^{k}v_{x}^{j} \frac{1}{2}h^{\ell m}(\partial_{x^{j}}(\lambda \omega_{m})-\partial_{x^{m}}(\lambda \omega_{j}))-\lambda v^{0} \omega_{\ell}\omega_{k}\nu_{x}^{k}v_{x}^{j}\frac{1}{2}h^{t\ell}\omega_{t}\partial_{x^{j}}\lambda \\
    &+\lambda (v^{0})^{2} \omega_{\ell} \nu_{x}^{i}\frac{1}{2}h^{\ell m}(\partial_{x^{i}}(\lambda \omega_{m})-\partial_{x^{m}}(\lambda \omega_{i}))-\lambda (v^{0})^{2} \omega_{\ell} \nu_{x}^{i}\frac{1}{2}h^{t\ell}\omega_{t}\partial_{x^{i}}\lambda \\
    & +\lambda v^{0} \omega_{\ell} \frac{1}{2}h^{\ell t}\omega_{t}(\partial_{x^{j}}(\lambda \omega_{i})+\partial_{x^{i}}(\lambda \omega_{j})) \nu_{x}^{i}v_{x}^{j} +\lambda v^{0} \omega_{\ell}^h \Gamma_{ij}^{\ell} \nu_{x}^{i}v_{x}^{j} \\
    & -\lambda v^{0} \omega_{\ell} \frac{1}{2}h^{\ell m} (\partial_{x^{j}}(\lambda \omega_{m}\omega_{i})+\partial_{x^{i}}(\lambda \omega_{m} \omega_{j})-\partial_{x^{m}}(\lambda \omega_{i} \omega_{j}) ) \nu_{x}^{i}v_{x}^{j} \\
    &+\lambda \omega_{a}v_{x}^{a}v_{x}^{j}\omega_{k}\partial_{x^{j}}\nu_{x}^{k}+\lambda \omega_{a}v_{x}^{a} v^{j}\nu_{x}^{k}\partial_{x^{j}}\omega_{k} \\
    &+\lambda \omega_{a}v_{x}^{a} v^{0} \omega_{k}\nu_{x}^{k}\frac{1}{2}h^{sm}\omega_{s}\partial_{x^{m}}\lambda \\
    &+\lambda \omega_{a}v_{x}^{a} \omega_{k}\nu_{x}^{k} v_{x}^{j}\frac{1}{2\lambda}\partial_{x^{j}}\lambda+\lambda \omega_{a}v_{x}^{a}\omega_{k}\nu_{x}^{k} v_{x}^{j} \frac{1}{2}h^{sm}\omega_{s}(\partial_{x^{j}}(\lambda \omega_{m})-\partial_{x^{m}}(\lambda \omega_{j})) \\
    & +\lambda v^{0}\omega_{a}v_{x}^{a} \nu_{x}^{i}\frac{1}{2\lambda}\partial_{x^{i}}\lambda+\lambda v^{0} \omega_{a}v_{x}^{a} \nu_{x}^{i} \frac{1}{2}h^{sm}\omega_{s}(\partial_{x^{i}}(\lambda \omega_{m})-\partial_{x^{m}}(\lambda \omega_{i})) \\
    &-\lambda \omega_{a}v_{x}^{a} \frac{1}{2\lambda}(\partial_{x^{j}}(\lambda \omega_{i})+\partial_{x^{i}}(\lambda \omega_{j}))\nu_{x}^{i}v_{x}^{j}+\lambda \omega_{a}v_{x}^{a}\omega_{s}\Gamma_{ij}^{s}\nu_{x}^{i}v_{x}^{j} \\
    &-\lambda \omega_{a}v_{x}^{a}\frac{1}{2}h^{sm}\omega_{s}(\partial_{x^{j}}(\lambda \omega_{m}\omega_{i})+\partial_{x^{i}}(\lambda \omega_{m}\omega_{j})-\partial_{x^{m}}(\lambda \omega_{i}\omega_{j}))\nu_{x}^{i}v_{x}^{j} \\
    &-\lambda \omega_{a}v_{x}^{a} \omega_{\ell} v_{x}^{j}\partial_{x^{j}}\nu_{x}^{\ell}-\lambda \omega_{a}v_{x}^{a} v^{0}\omega_{\ell}\omega_{k}\nu_{x}^{k}\frac{1}{2}h^{\ell m} \partial_{x^{m}} \lambda \\
    &-\lambda \omega_{a}v_{x}^{a}\omega_{\ell} \omega_{k}\nu_{x}^{k}v_{x}^{j} \frac{1}{2}h^{\ell m}(\partial_{x^{j}}(\lambda \omega_{m})-\partial_{x^{m}}(\lambda \omega_{j}))+\lambda \omega_{a}v_{x}^{a}\omega_{\ell}\omega_{k}\nu_{x}^{k}v_{x}^{j}\frac{1}{2}h^{t\ell}\omega_{t}\partial_{x^{j}}\lambda \\
    &-\lambda \omega_{a}v_{x}^{a} v^{0}\omega_{\ell} \nu_{x}^{i}\frac{1}{2}h^{\ell m}(\partial_{x^{i}}(\lambda \omega_{m})-\partial_{x^{m}}(\lambda \omega_{i}))+\lambda \omega_{a}v_{x}^{a} v^{0}\omega_{\ell} \nu_{x}^{i}\frac{1}{2}h^{t\ell}\omega_{t}\partial_{x^{i}}\lambda \\
    & -\lambda \omega_{a}v_{x}^{a} \omega_{\ell} \frac{1}{2}h^{\ell t}\omega_{t}(\partial_{x^{j}}(\lambda \omega_{i})+\partial_{x^{i}}(\lambda \omega_{j})) \nu_{x}^{i}v_{x}^{j} -\lambda \omega_{a}v_{x}^{a} \omega_{\ell}^h \Gamma_{ij}^{\ell} \nu_{x}^{i}v_{x}^{j} \\
    & +\lambda \omega_{a}v_{x}^{a} \omega_{\ell} \frac{1}{2}h^{\ell m} (\partial_{x^{j}}(\lambda \omega_{m}\omega_{i})+\partial_{x^{i}}(\lambda \omega_{m} \omega_{j})-\partial_{x^{m}}(\lambda \omega_{i} \omega_{j}) ) \nu_{x}^{i}v_{x}^{j} \\
    &+h_{k\ell}v_{x}^{k}\omega_{a}\nu_{x}^{a} v^{0}\frac{1}{2}h^{\ell m}\partial_{x^{m}}\lambda \\
    &+h_{k\ell}v_{x}^{k}v_{x}^{j}\omega_{a}\nu_{x}^{a}\frac{1}{2}h^{\ell m}(\partial_{x^{j}}(\lambda \omega_{m})-\partial_{x^{m}}(\lambda \omega_{j}))-h_{k\ell}v_{x}^{k}v_{x}^{j}\omega_{a}\nu_{x}^{a}\frac{1}{2}h^{t\ell}\omega_{t}\partial_{x^{j}}\lambda \\
    &+h_{k\ell}v_{x}^{k}\nu_{x}^{i}v^{0}\frac{1}{2}h^{\ell m}(\partial_{x^{i}}(\lambda \omega_{m})-\partial_{x^{m}}(\lambda \omega_{i}))-h_{k\ell}v_{x}^{k}\nu_{x}^{i}v^{0}\frac{1}{2}h^{t\ell}\omega_{t}\partial_{x^{i}}\lambda \\
    =&(^{h}\nabla_{v_{x}}\nu_{x},v_{x})_{h}-\lambda(v^{0}-\langle \omega,v_{x}\rangle)v_{x}^{j}\nu_{x}^{k}\partial_{x^{j}}\omega_{k} \\
    &-\lambda(v^{0}-\langle \omega,v_{x}\rangle) \langle \omega,\nu_{x} \rangle v_{x}^{j}\frac{1}{2\lambda}\partial_{x^{j}}\lambda-\lambda(v^{0}-\langle \omega,v_{x}\rangle)v^{0}\nu_{x}^{i}\frac{1}{2\lambda}\partial_{x^{i}}\lambda \\
    &+\lambda(v^{0}-\langle \omega,v_{x}\rangle)\frac{1}{2\lambda}( \partial_{x^{j}}(\lambda \omega_{i})+\partial_{x^{i}}(\lambda \omega_{j}) )\nu_{x}^{i}v_{x}^{j}\\
    &-\lambda(v^{0}-\langle \omega,v_{x}\rangle)\langle \omega,\nu_{x}\rangle \frac{1}{2}h^{t\ell}\omega_{\ell}\omega_{t}v_{x}^{j}\partial_{x^{j}}\lambda \\
    &-\lambda(v^{0}-\langle \omega,v_{x}\rangle)v^{0}\frac{1}{2}h^{t\ell}\omega_{t}\omega_{\ell}\nu_{x}^{i}\partial_{x^{i}}\lambda \\
    &+\lambda(v^{0}-\langle \omega,v_{x}\rangle)\frac{1}{2}h^{t\ell}\omega_{\ell}\omega_{t}( \partial_{x^{j}}(\lambda \omega_{i})+\partial_{x^{i}}(\lambda \omega_{j}))\nu_{x}^{i}v_{x}^{j} \\
    &+\langle \omega,\nu_{x} \rangle v^{0}\frac{1}{2}v_{x}^{k}\partial_{x^{k}}\lambda \\
    &+\langle \omega,\nu_{x}\rangle \frac{1}{2}(\partial_{x^{j}} (\lambda \omega_{k}) - \partial_{x^{k}}(\lambda \omega_{j})) v_{x}^{j}v_{x}^{k}-\langle \omega,v_{x}\rangle \langle \omega,\nu_{x} \rangle \frac{1}{2}v_{x}^{j}\partial_{x^{j}}\lambda \\
    &+v^{0}\frac{1}{2}(\partial_{x^{i}}(\lambda \omega_{k})-\partial_{x^{k}}(\lambda \omega_{i}))\nu_{x}^{i}v_{x}^{k}-v^{0}\langle \omega,v_{x}\rangle \frac{1}{2}\nu_{x}^{i}\partial_{x^{i}}\lambda.
\end{align*}

If we only consider vectors with $(\partial_{t},v)_{g}=\rho$, and we assume that $\langle\omega,v_{x} \rangle=0$ for all $v_{s} \in S^{k}\partial N$ (so that $v=(-\rho/\lambda,v_{x})$), we obtain
\[ (^g \nabla_{v}\nu,v)_{g}
    =(^{h}\nabla_{v_{x}}\nu_{x},v_{x})_{h}+(Yv_{x},\nu_{x})_{h}-dU(\nu_{x})+|\omega|_{h}^{2}dU(\nu)-|\omega|_{h}^{2}\langle d^{s}\rho\omega,\nu_{x}\otimes v_{x} \rangle, \]
where $Y$ corresponds to the Lorentz force associated to $\rho\omega$ and $U=\rho^{2}/(-2\lambda)$. Then,
\[ \Pi_{M}(v,v)
    =\Pi_{N}(v_{x},v_{x})-(Yv_{x},\tilde{\nu}_{x})_{h}+dU(\tilde{\nu}_{x})-|\omega|_{h}^{2}dU(\tilde{\nu}_{x})+|\omega|_{h}^{2}\langle d^{s}\rho\omega,\tilde{\nu}_{x}\otimes v_{x} \rangle, \]
where $\tilde{\nu}_{x}=-\nu_{x}$ is the interior normal to $\partial N$. Hence, we conclude:

\begin{lemma}
    Assume that $\langle \omega,v_{x}\rangle=0$ for all $v_{x} \in S^{k}\partial N$, and that 
    \[ \langle d^{s}\omega,\tilde{\nu}_{x}\otimes v_{x}\rangle-dU(\tilde{v}_{x})=0. \]
    Then $\partial N$ is strictly $\MP$-convex if and only if $\Pi(v,v)>0$ for all $v=(-\rho/\lambda,v_{x})$ with $v_{x} \in S^{k}\partial N$.
\end{lemma}

\subsection{Energy levels and angles} \label{sec:ener_ang}

As is mentioned in Appendix \ref{app:mp}, when we work with $\MP$-systems, we should consider energy levels strictly greater than the maximum of the potential. If we consider timelike geodesics $\gamma=(t,x)$ with $|\dot{\gamma}|_{g}^{2}=-m^{2}$, by Theorem \ref{thm:rel_lorentz_mp}, the energy level of the corresponding $\MP$-geodesic is $-m^{2}/2$. Therefore, for the $\MP$-system $(N,h,d\rho \omega,\rho^{2}(-2\lambda))$ the condition is
    \begin{equation} \label{eq:cond_lam}
        \lambda<\frac{\rho^{2}}{m^{2}}, \qquad \forall x \in N.
    \end{equation}
Note that from \eqref{eq:2e}, we already have 
    \begin{equation} \label{eq:cond_lam_partial}
        \lambda \leq \frac{\rho^{2}}{m^{2}}
    \end{equation}
on any $\MP$-geodesic, with equality if and only if $\dot{x}=0$ on some point of the geodesic $\gamma=(t,x)$. In this case, using \eqref{eq:const2} we obtain $\dot{t}=-m^{2}/\rho$ at that point. Hence, we still have to impose condition \eqref{eq:cond_lam} since it does not come from a fact of the manifold $(M,g)$.

If we let $B:=\max_{x \in M}\lambda$, then \eqref{eq:cond_lam} implies 
\begin{equation} \label{eq:bound_rho}
        \rho \in (-\infty,-m\sqrt{B}) \cup (m\sqrt{B},\infty).
\end{equation}

The boundness $\lambda$ has a geometric interpretation: we can find bounds for the hyperbolic angle between $\partial_{t}$ and the tangent vector to a timelike geodesic. Let $\gamma=(t,x)$ be a geodesic with $|\dot{\gamma}|_{g}^{2}=-m^{2}$ and $(\partial_{t},\dot{\gamma})_{g}=\rho$. Then, as in Remark \ref{rmk:timecone},
\begin{equation} \label{eq:angle1}
    \frac{\rho}{\mp m\sqrt{\lambda}}=\cosh \varphi,
\end{equation}
depending if $\dot{\gamma}$ is in the same timecone as $\partial_{t}$ ($\rho<0$) or not ($\rho>0$). Furthermore, if we write $A:=\min_{x \in M} \lambda$, in the case that $\dot{\gamma}$ is in the same timecone as $\partial_{t}$ ($\rho<0$) we have
\[ -\frac{\rho}{m\sqrt{B}} \leq \cosh \varphi \leq -\frac{\rho}{m\sqrt{A}}, \]
and using \eqref{eq:bound_rho} we get
\[ 1 <\cosh \varphi <\infty, \]
while if $\dot{\gamma}$ is in the opposite timecone as $\partial_{t}$ ($\rho>0$), a similar analysis gives
\[ 1 <\cosh \varphi <\infty. \]
Therefore, condition \eqref{eq:cond_lam} is equivalent to requiring the hyperbolic angle to be strictly greater than zero.

\subsection{The scattering relation} \label{sec:sca}

Let $\nu_{x}$ be the exterior unit normal to $\partial N$ (with respect to the metric $h$). Then, $\nu:=(\langle \omega,\nu_{x} \rangle,\nu_{x})$ is a unit normal exterior to $\R \times \partial N$. Hence, given $(t,x,v) \in TM$ with $(t,x) \in \R \times \partial M$, its projection to $T(\partial M)$ is 
\begin{align*}
    v'&=v-(v,\nu)_{g}\nu \\
    &=v-(v_{x},\nu_{x})_{h}\nu \\
    &=(v_{t}-\langle \omega,\nu_{x} \rangle(v_{x},\nu_{x})_{h},v_{x}-(v_{x},\nu_{x})_{h}\nu_{x}).
\end{align*}
Observe that $v_{x}'=v_{x}-(v_{x},\nu_{x})_{h}\nu_{x}$ is just the orthogonal projection of $v_{x}$ onto $T(\partial N)$. We will further decompose $v'$ in the following way. We write
\begin{equation} \label{eq:v'}
    v'=[v_{t}', v_{x}'],
\end{equation}
where $v_{t}=-(v', \partial_{t})_{g'}$, and $v_{x}'=(d\pi) v'$, where $\pi \colon M \to N$ is the projection on the second component. Since $v_{x}'$ is as before, there is no problem by using the same notation. The advantage of using \eqref{eq:v'}  is that $v'$ is written in an invariant way. 

\begin{thm} \label{thm:equivalence}
Let $(N, h)$ be a compact Riemannian manifold, let $\omega$ be a 1-form on $M$, let $\lambda>0$ be a function on $N$, and let $M=\R \times N$ be equipped with $g$ as in \eqref{eq:metric2}. Let $\rho$ be as in \eqref{eq:bound_rho}. Then we have the following.
\begin{enumerate}
    \item If
    \begin{equation} \label{eq:sca}
        \mathcal{S}_{\rho,m}(t, x,[v_{t} ', v_{x}'])=(s, y,[w_{t}', w_{x} ']),
    \end{equation}
    then
    \begin{equation} \label{eq:sca_mp}
        \mathcal{S}_{\MP}(x, v_{x}')=(y, w_{x}'),
    \end{equation}
    and,
    \begin{equation} \label{eq:act}
        \mathbb{A}(x, y)=\rho(t-s)-m^{2}T \quad \text { with } y=\exp _x^{\MP} v'
    \end{equation}
    where $v \in T_x M$ is the incoming timelike vector of mass $m$ with projection $v_{x}$, $\exp _x^{\MP}$ is the $\MP$-exponential map, $t=t(0)$, $s=t(T)$, and $\gamma \colon [0,T] \to M$ is the geodesic realizing $\mathcal{S}_{\rho}$ of energy $-m^{2}$ (see Definition \ref{defin:rho_sca}).
    \item Each one of the three quantities determines the other two: $\mathcal{S}_{\rho,m}, \mathcal{S}_{\MP}$, and $\mathbb{A}$.
\end{enumerate}
\end{thm}

Here $\mathcal{S}_{\MP}$ and $\A$ are the scattering relation and the boundary action function of energy $-m^{2}/2$ of the $\MP$-system $(N,h,d\rho \omega,\rho^{2}/(-2\lambda))$.

\begin{proof} \hfill
    \begin{enumerate}
    \item Let $v$ a vector of mass $m$ with projection $v'$ and let $\gamma$ the geodesic realizing \eqref{eq:sca}, that is, $\gamma$ goes from $(t,x)$ with velocity $v$ to $(s,y)$ with velocity $w$, and momentum $\rho$. Then, in light of Theorem \ref{thm:rel_lorentz_mp}, its spatial coordinate its a $\MP$-geodesic of energy $-m^{2}/2$ of the $\MP$-system $(N,h,d\rho \omega,\rho^{2}/(-2\lambda))$. Thus, $v_{x}'$ is the projection to $T\partial N$ of $v_{x}$, where the last vector has energy $-m^{2}/2$. By definition of the scattering relation for $\MP$-system, we conclude \eqref{eq:sca_mp}.
    
    For the action, assume that the geodesic $\gamma=(t,x)$ is defined in $[0,T]$. Then, from \eqref{eq:const2} we find
    \[ \rho \dot{t}-\langle \rho \omega,\dot{x} \rangle =\frac{\rho^{2}}{-\lambda}. \]
    This and \eqref{eq:2e} gives
    \begin{equation} \label{eq:act1}
        \rho \dot{t}-\langle \rho \omega,\dot{x} \rangle +\frac{\rho^{2}}{2\lambda}=-\frac{\rho^{2}}{2\lambda}
    \end{equation}
    Using \eqref{eq:2e} and \eqref{eq:act1}, we obtain
    \begin{equation} \label{eq:act2}
        \rho \dot{t}-\langle \rho \omega,\dot{x} \rangle +\frac{\rho^{2}}{2\lambda}=-\frac{m^{2}}{2}-\frac{|\dot{x}|_{h}^{2}}{2}.
    \end{equation}
    Rearranging terms in \eqref{eq:act2}, adding $-m^{2}/2$ to both sides, and integrating over from $0$ to $T$, we find \eqref{eq:act}.
    \item It is clear that from $\mathcal{S}_{\rho,m}$ we can obtain $\mathcal{S}_{\MP}$. That $\mathcal{S}_{\MP}$ and $\A$ determine each other is \cite{az}*{Theorem 4.3}. So, we only have to show that we can determine $\mathcal{S}_{\rho,m}$ from $\mathcal{S}_{\MP}$ and $\A$. To this end, let $(t,x,v') \in T\partial N$. Let $v$ with mass $m$ and projection to the boundary equal to $v'$. From \eqref{eq:sca_mp} we can find $y,w_{x}'$. If we consider the $\MP$-geodesic realizing \eqref{eq:sca_mp}, and define $t(\sigma)$ by \eqref{eq:const2}, with the initial condition $t(0)=t$, we obtain the geodesic $\gamma$ (in $(M,g)$) with momentum $\rho$ with initial position $(t,x)$ and initial velocity $v$. Since the parametrization of $\gamma$ is fixed, we obtain $w$, $s$ and $T$, which implies that we also have $w_{t}'$, and therefore, $S_{\rho,m}$ as well.
\end{enumerate}
\end{proof}

\section{Proof of the Theorems \ref{thm:intro1} and \ref{thm:intro2}} \label{sec:thms}

In this section we give the proofs of the main theorems presented in the introduction. We will always assume that $\rho$ satisfies \eqref{eq:cond_lam}. First, we need some preliminary results. The first one is the following observation.

\begin{lemma} \label{lemma:merging}
    The $\MP$-systems $(N,h,d\rho \omega,\rho^{2}/(-2\lambda))$ and $(N,h',d\rho \omega',\rho^{2}/(-2\lambda'))$ are $-m^{2}/2$-gauge equivalent if and only if, the $\MP$-systems $(N,h,d\omega,1/(-2\lambda))$ and $(N,h',d\omega',1/(-2\lambda'))$ are $-m^{2}/(2\rho^{2})$-gauge equivalent. Furthermore, 
\end{lemma}

We also have:

\begin{lemma} \label{lemma:simple_equiv}
    $(N,h,d\rho \omega,\rho^{2}/(-2\lambda))$ is simple with respect to curves of energy $-m^{2}/2$, if and only if, $(N,h,d\omega,1/(-2\lambda))$ is simple with respect to curves of energy $-m^{2}/(2\rho^{2})$.
\end{lemma}

\begin{proof}
In first place, observe that the energy functions which correspond to the systems $(N,h,d\rho \omega,\rho^{2}/(-2\lambda))$ and $(N,h,d\omega,1/(-2\lambda))$ are
\[ E_{\rho}(x,\theta)=\frac{1}{2}|\theta|^{2}+\frac{\rho^{2}}{-2\lambda}, \quad E(x,\xi)=\frac{1}{2}|\xi|^{2}+\frac{1}{-2\lambda}, \]
respectively. Then, $(x,\xi) \in E^{-1}(-m^{2}/(2\rho^{2}))$ if and only if $(x,\rho \xi) \in E_{\rho}^{-1}(-m^{2}/2)$. Hence, is clear that from \eqref{eq:stric_mp_con} we can obtain strictly convexity for $(N,h,d\rho \omega,\rho^{2}/(-2\lambda))$ by multiplying by $\rho^{2}$. Since this procedure is reversible, the equivalence for the strict convexity follows. 

Before dealing with the exponential maps, note that from our observation about the energy functions and \eqref{eq:mp-geo}, $\sigma(\tau)$ is a $\MP$-geodesic of energy $-m^{2}/(2\rho^{2})$ for the system $(N,h,d\omega,1/(-2\lambda))$ if and only if $\zeta=\sigma(\rho \tau)$ is a $\MP$-geodesic of energy $-m^{2}/2$ of the system $(N,h,d\rho \omega,\rho^{2}/(-2\lambda))$. From this it follows that the $\MP$-exponential map of energy $-m^{2}/(2\rho^{2})$ for the system $(N,h,d\omega,1/(-2\lambda))$ if a diffeomorphism if and only if the same is true for $\MP$-exponential map of energy $-m^{2}/2$ of the system $(N,h,d\rho \omega,\rho^{2}/(-2\lambda))$. This finishes the proof.
\end{proof}

\begin{proof}[Proof of Theorem \ref{thm:intro1}]
 Let $\mathcal{S}_{\MP_{\rho}}$ and $\mathcal{S}_{\MP_{\rho}}'$ be the scattering relation of the systems $(N,h,d\rho \omega,\rho^{2}/(-2\lambda))$ and $(N,h',d\rho \omega',\rho^{2}/(-2\lambda'))$ at energy level $-m^{2}/2$, respectively. From Theorem \ref{thm:equivalence}, we have that $\mathcal{S}_{\MP_{\rho}}=\mathcal{S}_{\MP_{\rho}}'$. Hence, the result assuming that we only have information for one $\rho$, follows from Theorem \ref{thm:MP_rig}.    

For the second part, again by Theorem \ref{thm:equivalence}, we have that $\mathcal{S}_{\MP_{\rho_{i}}}=\mathcal{S}_{\MP_{\rho_{i}}}'$ for $i=1,2$ and energy $-m^{2}/2$. Then, by Theorem \ref{thm:MP_rig_2energy}, the systems $(N,h,d\rho_{i} \omega,\rho_{i}^{2}/(-2\lambda))$ and $(N,h',d\rho_{i} \omega',\rho_{i}^{2}/(-2\lambda'))$ are $-m^{2}/2$-gauge equivalent. By Lemmas \ref{lemma:merging}, \ref{lemma:simple_equiv}, we conclude that $(N,h,d\omega,1/(-2\lambda))$ and $(N,h',d\omega',1/(-2\lambda'))$ are $-m^{2}/(2\rho_{i}^{2})$-gauge equivalent and simple with respect to curves of energy $-m^{2}/(2\rho_{i}^{2})$. In light of \cite{mt23}*{Lemma 4.4}, these systems have the same $\MP$ scattering relation for two levels of energy. Hence, by Theorem \ref{thm:MP_rig_2energy}, the systems $(N,h,d\omega,1/(-2\lambda))$ and $(N,h',d\omega',1/(-2\lambda'))$ are gauge equivalent, which gives \eqref{eq:ssm_mrho_equiv_intro} with $\mu=1$.
\end{proof}

The proof of Theorem \ref{thm:intro2} is a verbatim of the proof of Theorem \ref{thm:intro1}, using Theorem \ref{thm:MP_gen} instead of Theorem \ref{thm:MP_rig}, and Theorem \ref{thm:MP_gen_2e} instead of Theorem \ref{thm:MP_rig_2energy}.

\begin{rmk}
Note that if $g=g_{h,\omega,\lambda}$ and $g'=g_{h',\omega',\lambda'}$, with $h|_{\partial M}=h'|_{\partial M}$, $i^{*}\omega=i^{*}\omega'$ and $\lambda|_{\partial M}=\lambda'|_{\partial M}$, $g$ and $g'$ as in \eqref{eq:ssm_mrho_equiv_intro}, then $\mathcal{S}_{\rho,m}=\mathcal{S}_{\rho,m}'$. Indeed, if $g'$ has that form, then $(N,h,d\rho \omega,\rho^{2}/(-2\lambda))$ and $(N,h',d\rho \omega',\rho^{2}/(-2\lambda'))$ are $-m^{2}/2$-gauge equivalent as $\MP$-systems. By \cite{mt23}*{Lemma 4.4}, $\mathcal{S}_{\MP}=\mathcal{S}_{\MP}'$. Hence, by Theorem \ref{thm:equivalence}, we obtain $\mathcal{S}_{\rho,m}=\mathcal{S}_{\rho,m}$.
\end{rmk}

\begin{rmk} \label{rmk:sharp}
    We would like to point out that the results are sharp in the following sense. We claim that if $g=g_{h,\omega,\lambda}$ and $g'=g_{h',\omega',\lambda'}$ with $h,h',\omega,\omega',\lambda,\lambda'$ as in \eqref{eq:ssm_mrho_equiv_intro} with $\mu=1$, that is, 
    \begin{equation} \label{eq:equiv_ssm}
        h'=f^{*}h, \quad \omega'=f^{*}\omega+d\varphi, \quad \frac{1}{\lambda'}=\frac{1}{f^{*}\lambda},
    \end{equation}
    and $h|_{\partial M}=h'|_{\partial M}$, $i^{*}\omega=i^{*}\omega'$ and $\lambda|_{\partial M}=\lambda'|_{\partial M}$, then $\mathcal{S}_{\rho,m}=\mathcal{S}_{\rho,m}'$ for all $\rho$. Indeed, the associated $\MP$-systems are $(N,h,d\rho \omega,\rho^{2}/(-2\lambda))$ and $(N,h,d\rho \omega',\rho^{2}/(-2\lambda'))$. By \eqref{eq:equiv_ssm}, these systems are $-m^{2}/2$-gauge equivalent with $\mu=1$ (see Definition \ref{defin:equiv_mp}). Then, by \cite{mt23}*{Lemma 4.4}, the systems have the same scattering relation. By Theorem \ref{thm:equivalence}, we conclude that $\mathcal{S}_{\rho,m}=\mathcal{S}_{\rho,m}'$. Since this is valid for every $\rho$, we obtain the result.
\end{rmk}

Finally we give an interpretation of results on Theorem \ref{thm:intro1} and \ref{thm:intro2} in terms of metrics of the form \eqref{eq:metric2}: on standard stationary manifolds, the scattering relation, even for more that two different values of $\rho$, is not capable of distinguish between $g=g_{h,\omega,\lambda}$ and $\Psi^{*}g$, where 
\begin{align*}
    \Psi \colon M &\to M, \\
    (t,x) & \mapsto (t+\varphi(x),f(x)),
\end{align*}
where $f \colon N \to N$ is a diffeomorphism with $f|_{\partial M}=id|_{\partial M}$ and $\varphi$ is a smooth function vanishing on the boundary of $\partial N$.

\section{MP-systems} \label{app:mp}

In this appendix we recall some notions and results from \cite{az}, \cite{mt23}, and \cite{mt24}. 

An $\MP$-system consist of a smooth compact Riemannian manifold with smooth boundary $(N,h)$, a closed 2-form $\Omega$, and a smooth function $U$. Given an $\MP$-system $(N,h,\Omega,U)$, the magnetic field $\Omega$ induces a map $Y \colon TN \to TN$ given by
\[ \Omega_{x}(u,v)=(Y_{x}u,v)_{h}, \]
where $u,v \in T_{x}N$. This map is usually called \emph{Lorentz force} associated to the magnetic field $\Omega$. $C^{2}$ curves $\sigma \colon [a,b] \to N$ that satisfy
\begin{equation} \label{eq:mp-geo}
    \nabla_{\dot{\sigma}} \dot{\sigma}=Y(\dot{\sigma})-\nabla U(\sigma),
\end{equation}
are called \emph{$\MP$-geodesics}. Here, $\nabla$ is the Levi-Civita connection associated to the metric $h$. Equation \eqref{eq:mp-geo} defines a flow , called the \emph{$\MP$-flow}, and given by
\[ \phi_{t}(x,v)=(\sigma(t),\dot{\sigma}(t)),  \]
where $\sigma$ solves \eqref{eq:mp-geo} and $\sigma(0)=x$, $\dot{\sigma}(0)=v$. For the $\MP$-flow the energy $E \colon TN \to \R$ given by $E(x,v)=\frac{1}{2}|v|_{g}^{2}+U(x)$ is an integral of motion. Curves satisfying  \eqref{eq:mp-geo} and with $E(\sigma,\dot{\sigma}) \equiv k$ are called $\MP$-geodesics of energy $k$. Given $k \in \R$, we define $S^{k}N:=\{E=k\}$. We will assume always that $k>\max_{M} U$. Let $\tilde{\nu}(x)$ be the inward unit vector normal to $\partial N$ at $x$, and set
\[
\partial_{\pm} S^{k}N:=\{(x, v) \in S^{k}N: x \in \partial N, \pm (v, \tilde{\nu}(x))_{h(x)} \geq 0\}.
\]
We denote by $\Pi$ the second fundamental form. We say that $\partial M$ is \emph{strictly $\MP$-convex} if
\begin{align} \label{eq:stric_mp_con}
    \Pi(x, \xi)>(Y_x(\xi), \tilde{\nu}(x))_{h}-d_x U(\nu(x))
\end{align}
for all $(x, \xi) \in S^k(\partial M)$. For $x \in N$, the $\MP$-exponential map at $x$ of energy $k$ is the map $\exp _x^{k} \colon T_x N \to N$ given by
\[
\exp _x^{k}(t v)=\pi \circ \phi_t(v),
\]
where $t \geq 0, v \in S_x^k N$, and $\pi: T M \rightarrow M$ is the base point projection.

\begin{defin} \label{defin:simple}
We say that $N$ is ($\MP$-)\emph{simple} with respect to $(h, \Omega, U)$ if $\partial N$ is strictly $\MP$-convex and the $\MP$-exponential map $\exp_{x}^{k} \colon (\exp_x^{k})^{-1}(N) \to N$ is a diffeomorphism for every $x \in N$.   
\end{defin}

In this case, $\Omega$ is exact, that is, there exist a 1-form $\alpha$ such that $\Omega=d\alpha$, and we call $\alpha$ to be the magnetic potential. Henceforth we call $(g, \alpha, U)$ a simple $\MP$-system on $M$. We will also say that $(M, g, \alpha, U)$ is a simple $\MP$-system.

For $(x, v) \in \partial_{+} S^{k}N$, let $\tau(p, v)$ be the time exit time function for the $\MP$-geodesic $\sigma$ with $\sigma(0)=p$, $\dot{\sigma}(0)=v$. By \cite{az}*{Lemma~A.3} we have that for a simple $\MP$-system, the function $\tau \colon \partial_{+} S^{k}N \to \R$ is smooth.

\begin{defin} \label{defin:scattering}
First consider the map $\hat{\mathcal{S}} \colon \partial_{+} S^{k}N \to \partial_{-} S^{k}N$ for an $\MP$-system $(N, g, \alpha, U)$ is defined as
\[
\hat{\mathcal{S}}(p,v)=(\phi_{\tau (p, v)}(p, v))=(\sigma(\tau(p,v)), \dot{\sigma}(\tau(p,v))).
\]    
The \emph{scattering relation} relation of energy $k$ of this system is 
\[ \mathcal{S}_{\MP}(p,v')=(\sigma(\tau(p,v)), \dot{\sigma}(\tau(p,v))'), \]
where the $'$ denotes the projection of the vector to $T\partial M$.
\end{defin}

\begin{rmk} \label{rmk:sca_proj}
Note that this scattering relation is a slightly different to the one presented in \cites{az, mt23, mt24}, because of the projection on the boundary. Defining the scattering in this way make it invariant under the group of transformations from Definition \ref{defin:equiv_mp} below. More specifically, the diffeomorphism appearing in the group of transformations could possible change the normal vector, so, projecting to the boundary, we avoid that problem. In this way is defined in the Riemannian setting \cites{su08, su08.2, stefanov08} and in the Lorentzian setting \cites{plamen, plamen24}.
\end{rmk}

It has been shown that $\MP$-geodesics minimize the \emph{time free action} of energy $k$ (see \cite{az}*{Appendix A.1})
\[ \A(\sigma)=\frac{1}{2} \int_{0}^{T}|\dot{\sigma}|_{h}^{2}dt+kT-\int_{0}^{T}(\alpha(\sigma(t),\dot{\sigma}(t))+U(\sigma(t)) )dt, \]
where 
\[ \sigma \in \mathcal{C}(x,y)=\{\sigma \in AC([0,T],N): \sigma(0)=x \text{ and } \sigma(T)=y\}, \]
so that the \emph{Ma\~n\'e action potential} (of energy $k$) is well defined
\begin{equation} \label{eq:mane_act}
    \A(x,y)=\inf_{\gamma \in \mathcal{C}(x,y)}\A(\gamma).
\end{equation}
The \emph{boundary action function} (of energy $k$) is defined as
\begin{equation} \label{eq:bda}
    \A|_{\partial N \times \partial N}
\end{equation}

\begin{defin} \label{defin:equiv_mp} \hfill 
\begin{enumerate}
    \item We say that two $\MP$-systems $(h, \alpha, U)$ and $(h', \alpha', U')$ are $k$-\emph{gauge equivalent} if there is a diffeomorphism $f\colon N \to N$ with $f|_{\partial N}=id_{\partial N}$, a smooth function $\varphi \colon N \to \R$ with $\varphi|_{\partial N}=0$, and a strictly positive function $\mu \in C^{\infty}(N)$, such that $h'=\frac{1}{\mu}f^{*}h$, $\alpha'=f^* \alpha+d \varphi$ and $U'=\mu (f^{*}U-k)+k$. 
    \item We say that two $\MP$-systems $(h, \alpha, U)$ and $(h', \alpha', U')$ are \emph{gauge equivalent} if there is a diffeomorphism $f\colon N \to N$ with $f|_{\partial N}=id_{\partial N}$, a smooth function $\varphi \colon N \to \R$ with $\varphi|_{\partial N}=0$, so that $h'=f^{*}h$, $\alpha'=f^* \alpha+d \varphi$ and $U'=f^{*}U$. 
\end{enumerate}
\end{defin}

We also have the following theorem, which is a combination of the results in \cite{mt23}.

\begin{thm} [\cite{mt23}] \label{thm:MP_rig}
Let $(h, \alpha, U)$ and $(h', \alpha', U')$ be simple $\MP$-systems on $N$ with the same boundary action functions for energy $k$. Assume one of the following conditions
\begin{enumerate}
    \item $h'$ is conformal to $h$;
    \item The $\MP$-systems are real-analytic;
    \item $\dim N=2$.
\end{enumerate}
Then, the $\MP$-systems are $k$-gauge equivalent. Furthermore, in (1) we have that the diffeomorphism is just $id_{M}$. The same is true if we require the systems to have the same scattering relation of energy $k$ instead of the same boundary action function of energy $k$, with the additional assumption that $h|_{\partial M}=h'|_{\partial M}$, $i^* \omega=i^* \omega'$, $\lambda|_{\partial M}=\lambda'|_{\partial M}$, where $i$ is the embedding $i \colon \partial M \to M$.
\end{thm}

Assuming the knowledge of the boundary action function for two energy levels, Y. Assylbekov and H.Zhou obtained the following result:

\begin{thm}[\cite{az}] \label{thm:MP_rig_2energy}
   Let $(h, \alpha, U)$ and $(h', \alpha', U')$ be simple $\MP$-systems on $N$ with the same boundary action functions for two energy levels. Assume one of the following conditions
\begin{enumerate}
    \item $h'$ is conformal to $h$;
    \item The $\MP$-systems are real-analytic;
    \item $\dim N=2$.
\end{enumerate}
Then, the $\MP$-systems are gauge equivalent. Furthermore, in (1) we have that the diffeomorphism is just $id_{M}$. The same is true if we require the systems to have the same scattering relation of energy for two energy levels.
\end{thm}

\begin{rmk}
On the hypothesis of Theorem \ref{thm:MP_rig_2energy}, we do not need that the quantities agree on the boundary because this is already satisfied, since the boundary action functions agree on two energy levels, see \cite{az}*{Lemma 4.1}.
\end{rmk}

\begin{rmk}
    On hypothesis about simplicity on theorems \ref{thm:MP_rig} and \ref{thm:MP_rig_2energy}, one, of course, only requires the conditions from Definition \ref{defin:simple} for the energy levels that are involved.
\end{rmk}

Let $I$ be the $\MP$-ray transform, that is, the linearization of the $\MP$-boundary action function \eqref{eq:bda}. For a fixed manifold $N$, we define $\mathcal{G}^{m}$ to be the set of simple $C^m$ systems $(h, \alpha,U)$ with s-injective $\MP$-ray transform $I$. The author obtained the following generic local boundary rigidity result for $\MP$-systems:

\begin{thm}[\cite{mt24}*{Theorem 7.5}] \label{thm:MP_gen}
    There exists $m$ big enough such that for every $(h_0, \alpha_0, U_{0}) \in \mathcal{G}^m$, there is $\eps>0$ such that for any two $\MP$-systems $(h, \alpha, U)$, $(h, \alpha,U)$ with
    \begin{align*}
        \|h-h_0\|_{C^m(N)}+\|\alpha-\alpha_{0}\|_{C^m(N)}+\|U-U_{0}\|_{C^{m}(N)} &\leq \eps, \\
        \|h'-h_0\|_{C^m(N)}+\|\alpha'-\alpha_{0}\|_{C^m(N)}+\|U'-U_{0}\|_{C^{m}(N)} &\leq \eps,
    \end{align*}
    we have the following:
    \[\mathbb{A}_{h, \alpha,U}=\mathbb{A}_{h', \alpha',U'} \quad \text { on } \partial N \times \partial N, \]
    (where the boundary action functions are of energy $k$), implies that $(h, \alpha,U)$ and $(h', \alpha',U')$ are $k$-gauge equivalent.
\end{thm}

Since for simple $\MP$-systems, $\A$ and $\mathcal{S}$ determine each other (\cite{az}*{Theorem 4.3}), the same result is true if we change $\A$ by $\mathcal{S}$. We can improve the results from Theorem \ref{thm:MP_gen} by assuming equality of the scattering relation for two levels of energy:

\begin{thm} \label{thm:MP_gen_2e}
    Assume that $h|_{\partial M}=h'|_{\partial M}$, $U|_{\partial M}=U'|_{\partial M}$ and $i^* \alpha=i^* \alpha'$ (where $i$ is the embedding $i \colon \partial M \to M$). There exists $m$ big enough such that for every $(h_0, \alpha_0, U_{0}) \in \mathcal{G}^m$, there is $\eps>0$ such that for any two $\MP$-systems $(h, \alpha, U)$, $(h, \alpha,U)$ with
    \begin{align*}
        \|h-h_0\|_{C^m(N)}+\|\alpha-\alpha_{0}\|_{C^m(N)}+\|U-U_{0}\|_{C^{m}(N)} &\leq \eps, \\
        \|h'-h_0\|_{C^m(N)}+\|\alpha'-\alpha_{0}\|_{C^m(N)}+\|U'-U_{0}\|_{C^{m}(N)} &\leq \eps,
    \end{align*}
    we have the following:
    \[\mathcal{S}_{h, \alpha,U}=\mathcal{S}_{h', \alpha',U'}, \]
    for two energy levels, implies that $(h, \alpha,U)$ and $(h', \alpha',U')$ are gauge equivalent.
\end{thm}

\begin{proof}
    By Theorem \ref{thm:MP_gen} there is a diffeomorphism $f\colon N \to N$ with $f|_{\partial N}=id_{\partial N}$, a smooth function $\varphi \colon N \to \R$ with $\varphi|_{\partial N}=0$, and a strictly positive function $\mu \in C^{\infty}(N)$, such that $h'=\frac{1}{\mu}f^{*}h$, $\alpha'=f^* \alpha+d \varphi$ and $U'=\mu (f^{*}U-k)+k$. Note that this already gives the relation between $\alpha$ and $\alpha'$ that we need. Our work now focuses in proving that $\mu=1$ (note that we already have $\mu|_{\partial M}=1$). Let $\tilde{\mathcal{S}}$ be scattering relation associated to $(f^{*}h,\alpha',f^{*}U)$. The fact that $(f^{*}h,\alpha',f^{*}U)$ and $(h,\alpha,U)$ are gauge equivalent imply that they have the same scattering relation (\cite{mt23}*{Lemma 4.4}). Since $h'$ is conformal to $f^{*}h$, Theorem \ref{thm:MP_rig_2energy} (1) implies that $\mu=1$, that is $h=f^{*}h$ and $U'=f^{*}U$.
\end{proof}

\section{Reduction of Hamiltonian systems with symmetry} \label{app:ham_red}

\begin{defin} \label{def:mo} Let $(P,\eta)$ be a symplectic manifold, $\Phi \colon G \times P \to P$ be a symplectic action of the Lie group $G$ on $P$.
    \begin{enumerate}
        \item A map 
    \[ J \colon P \to \mathfrak{g}^{*}, \]
    is a \emph{momentum mapping} for the action $\Phi$, if for every $\xi \in \mathfrak{g}$, we have
    \[ d(\hat{J}(\xi))=\iota_{\xi_{X}}\eta, \]
    where
    \begin{align*}
        \hat{J}(\xi) \colon P  &\to \R, \\
        x &\mapsto J(x) \xi,
    \end{align*}
    and $\xi_{X}$ is the infinitesimal generator of the action of $\R$ corresponding to the subgroup $\exp(t\xi)$. 
    \item A momentum map $J$ is said to be \emph{$A d^*$-equivariant} if the actions are compatible with $J$, i.e. $J(\Phi_g(x))=Ad_{g^{-1}}^* J(x)$ $ \forall x \in X$, $\forall g \in G$.
    \end{enumerate}
     \item The tuple $(P,\omega,J,\Phi)$ is called a \emph{Hamiltonian $G$-space}.
\end{defin}

To find a momentum map in Section \ref{sec:ham_red}, we use the following result.

\begin{lemma}[\cite{am}*{Corollary 4.2.13}] \label{lemma:mo_tan}
    Let $M$ be a pseudo-Riemannian manifold and let $G$ be a group acting on $M$ by isometries. Lift this to a symplectic action on $TM$. Then, its momentum mapping is given by
    \begin{equation} \label{eq:mom}
        \hat{J}(\xi)(v_q)=(v_q, \xi_Q(q))_{g},
    \end{equation}
    and is $Ad^{*}$-equivariant.
\end{lemma}

The main tools to relate geodesics of standard stationary manifolds with $\MP$-geodesics are the following versions of the reduction theorem due Marsden and Weinstein \cite{mw} and Meyer \cite{meyer}, see also \cites{mr, mmopr}.

\begin{thm}[\cite{am}*{Theorem 4.3.1}] \label{thm:mwm}
    Let $(P, \eta)$ be a symplectic manifold on which the Lie group $G$ acts symplectically and let $J \colon P \to \mathfrak{g}^*$ be an $A d^*$-equivariant momentum mapping for this action. Assume $\rho \in \mathfrak{g}^*$ is a regular value of $J$ and that the isotropy group $G_\rho$ under the $Ad^{*}$ action on $\mathfrak{g}^*$ acts freely and properly on $J^{-1}(\rho)$. 
    Then $P_\rho=J^{-1}(\rho) / G_\rho$ has a unique symplectic form $\eta_\rho$ with the property
    \[ \pi_\rho^* \eta_\rho=i_\rho^* \eta, \]
    where $\pi_\rho \colon J^{-1}(\rho) \to P_\rho$ is the canonical projection and $i_\rho: J^{-1}(\rho) \to P$ is the inclusion.
\end{thm}

\begin{thm}[\cite{am}*{Theorems 4.3.5}]  \label{thm:ham_red}
    Under the assumptions of Theorem \ref{thm:mwm}, if $H \colon P \to \R$ is invariant under the action of $G$, then the flow $F_t$ of $X_H$ leaves $J^{-1}(\rho)$ invariant and commutes with the action of $G_\rho$ on $J^{-1}(\rho)$, so it induces canonically a flow $H_t$ on $P_\rho$ satisfying 
    \[\pi_\rho \circ F_t=H_t \circ \pi_\rho.\]
    This flow is a Hamiltonian flow on $P_\rho$ with $a$ Hamiltonian $H_\rho$ satisfying 
    \[H_\rho \circ \pi_\rho=H \circ i_\rho.\]
    $H_\rho$ is called the \emph{reduced Hamiltonian}.
\end{thm} 

\section{The geometry of the lightlike case} \label{app:light_geo}

In this section, we revisit some hypothesis in the work \cite{plamen}, in order to relate the geometry of $(M,g)$ with $g$ as in \eqref{eq:metric2} with $\lambda=1$ with a corresponding magnetic system. 

From Theorem \ref{thm:rel_lorentz_mp}, we obtain the following.

\begin{cor}
    Let $(M,g)$ be a standard stationary space time with $\lambda=1$. 
    \begin{enumerate}
        \item The symplectic reduction of the Hamiltonian $G$-space $(TM,\eta,\Phi,J)$ endowed with the Hamiltonian $H$ is given by the manifold $TN$  endowed with the symplectic form 
        \[  \eta_{\rho}=\delta^{\flat}+\pi^{*}(\rho d\omega), \] where $\delta^{\flat}$ is the pullback of the canonical 2-form $\delta$ on $T^{*}N$ to $TN$ via $h$, and $\pi \colon TN \to N$ is the base point projection. The reduced Hamiltonian is given by
        \[ H_{\rho}(x,v)=\frac{1}{2}|v_{x}|_{h(x)}^{2}+\frac{\rho^{2}}{-2}. \]
        \item The projection of every trajectory of the geodesic flow on $(TM,g)$ corresponds to a trajectory of the Hamiltonian flow on $(TN,\eta_{\rho},H_{\rho})$ for some $\rho \in \R$. Reciprocally, given a trajectory $(x,v)$ of the Hamiltonian flow on $(TN,\eta_{\rho},H_{\rho})$, one can find a trajectory of the flow on $(TM,\eta,H)$ whose projection to $TN$ is $(x,v)$.
        \item The projection of every geodesic in $(M,g)$ corresponds an $\MP$-geodesic on the $\MP$-system $(N,h,\rho d\omega,\rho^{2}/(-2))$, where $\rho \in \R$. Reciprocally, given an $\MP$-geodesic $x$ of the $\MP$-system $(N,h,\rho d\omega,\rho^{2}/(-2))$, one can find a geodesic on $(M,g)$ whose projection to $N$ is $x$. In particular, lightlike geodesics correspond to curves with $|\dot{x}|_{h}^{2}=\rho^{2}$.
    \end{enumerate}
\end{cor}

Note that in contrast to what happens in \cite{plamen}, we obtain an $\MP$-system and not a magnetic system. However, the potential in the $\MP$-system that we just found is constant. Hence, the $\MP$-geodesics with energy 0 are magnetic geodesics of the system $(N,h,d\rho \omega)$ with $|\dot{x}|_{h}^{2}=\rho^{2}$. So, this last magnetic system is the one that appears in \cite{plamen}. In that work, the following definition is given

\begin{defin} \label{defin:simple_mag}
    A stationary space $(M,g)$ is simple if the ``projected'' magnetic system is simple.
\end{defin}

Of course, we have a lot of magnetic systems, one for each $\rho$. However, in that article the author consider lightlike geodesics, which satisfy
\begin{equation} \label{eq:rho_light}
    \dot{t}-\langle \omega,\dot{x} \rangle =1.
\end{equation}
Since the reparametrization of a lightlike geodesic is again a lightlike geodesic, we can chose a parametrization of $\gamma$ so that $\rho=-1$. Therefore, magnetic system from Definition \ref{defin:simple_mag} is just $(N,h,-d\omega)$.

\begin{lemma}
    Let $p=(t_{1},y)$, $q=(t_{2},z)$ with $y \neq z$. Let $\ell=\{(s,z):s \in \R\}$. Then, there exists a unique lightlike geodesic joining $p$ with $\ell$ so that $(\partial_{t},\dot{\gamma})_{g}=-1$ if and only if the magnetic exponential map $\exp_{y}^{\mu}$ of the system $(N,h,-d\omega)$ is a diffeomorphism.
\end{lemma}

\begin{proof}
Let $\gamma=(t,x)$ be the unique lightlike geodesic joining $p$ with $\ell$ with $(\partial_{t},\dot{\gamma})_{g}=1$. Then, $x$ is a unit speed magnetic geodesic of the system $(N,h,-d\omega)$ with $x(0)=y$ and $x(s_{0})=z$, for some $s_{0}$. Then, $\exp_{y}^{\mu}(s_{0}\dot{x}(0))=z$, which shows that $\exp_{y}^{\mu}$ is surjective. If we have $\exp_{y}^{\mu}(s_{1}v_{1})=\exp_{y}^{\mu}(s_{2}v_{2})=z$, then there are two unit speed magnetic geodesics $x_{1},x_{2}$, so that $x_{j}(0)=y$, $\dot{x}_{1}(0)=v_{1}$, $\dot{x}_{2}=v_{2}$ and $x_{1}(s_{1})=x_{2}(s_{2})=z$. Define $t$ by \eqref{eq:rho_light} (with $\rho=-1$) and initial condition $t(0)=t_{1}$. Then $\gamma_{j}:=(t,x_{j})$ are two lightlike geodesics joining $z$ with $\ell$ so that $(\partial_{t},\dot{\gamma})_{g}=-1$, which contradicts the uniqueness condition. This shows injectivity of the exponential map. Since the magnetic exponential map is a local diffeomorphism, we obtain the desired conclusion.

Reciprocally, let $p=(t_{1},y)$, $q=(t_{2},z)$ with $y \neq z$ be points in $M$. Since $\exp_{y}^{\mu}$ is a diffeomorphism, there exist a unique unitary vector $v_{x}$ and a unique $s_{0}$ so that $\exp_{y}^{\mu}(s_{0}v_{x})=z$. This means that there exists a unit speed magnetic geodesic $x$ so that $x(0)=y$, $\dot{x}(0)=v_{x}$ and $x(s_{0})=z$. Define $t$ by \eqref{eq:rho_light} (again with $\rho=-1$) and initial condition $t(0)=t_{1}$. Therefore, $\gamma:=(t,x)$ is a lightlike geodesic joining $z$ with $V$ so that $(\partial_{t},\dot{\gamma})_{g}=-1$. If there exists another $\gamma$ satisfying the previous conditions, say $\gamma=(\tilde{t},\tilde{x})$, then $\tilde{x}$ would satisfies the same conditions as $x$, contradicting the injectivity of the exponential map.
\end{proof}

In our setting, \cite{masiello}*{Theorem 6.3.4} implies the existence of lightlike geodesic joining points with integral lines of $\partial_{t}$. Therefore, the only ``new'' condition imposed by Definition \ref{defin:simple_mag} is that $\partial N$ is convex in the magnetic sense.

Using our computations from Section \ref{sec:conv} with $\lambda=1$, we see
\begin{align*}
    (^g \nabla_{v}\nu,v)_{g}
    =&(^{h}\nabla_{v_{x}}\nu_{x},v_{x})_{h}-(v^{0}-\langle \omega,v_{x}\rangle)v_{x}^{j}\nu_{x}^{k}\partial_{x^{j}}\omega_{k} \\
    &+(v^{0}-\langle \omega,v_{x}\rangle)\frac{1}{2}( \partial_{x^{j}}\omega_{i}+\partial_{x^{i}}\omega_{j} )\nu_{x}^{i}v_{x}^{j}\\
    &+(v^{0}-\langle \omega,v_{x}\rangle)\frac{1}{2}h^{t\ell}\omega_{\ell}\omega_{t}( \partial_{x^{j}}\omega_{i}+\partial_{x^{i}}\omega_{j})\nu_{x}^{i}v_{x}^{j} \\
    &+\langle \omega,\nu_{x}\rangle \frac{1}{2}(\partial_{x^{j}} \omega_{k} - \partial_{x^{k}}\omega_{j}) v_{x}^{j}v_{x}^{k}\\
    &+v^{0}\frac{1}{2}(\partial_{x^{i}}\omega_{k}-\partial_{x^{k}} \omega_{i})\nu_{x}^{i}v_{x}^{k}.
\end{align*}
If we only consider vectors with $(\partial_{t},v)_{g}=-1$, and we assume that $\langle\omega,v_{x} \rangle=0$ for all $v \in S\partial N$ (so that $v=(1,v_{x})$), we obtain
\[ (^g \nabla_{v}\nu,v)_{g}
    =(^{h}\nabla_{v_{x}}\nu_{x},v_{x})_{h}+(Yv_{x},\nu_{x})_{h}+|\omega|_{h}^{2}\langle d^{s}\omega,\nu_{x}\otimes v_{x} \rangle, \]
where $Y$ corresponds to the Lorentz force associated to $-\omega$. Then,
\[ \Pi_{M}(v,v)
    =\Pi_{N}(v_{x},v_{x})-(Yv_{x},-\nu_{x})_{h}+|\omega|_{h}^{2}\langle d^{s}\omega,\nu_{x}\otimes v_{x} \rangle, \]
Our computations can now be stated as follows.

\begin{lemma}
    Assume that $d^{s}\omega \equiv 0$ and $\langle \omega,v_{x}\rangle=0$ for all $v_{x} \in S\partial N$. Then $\partial N$ is strictly magnetic convex if and only if $\Pi(v,v)>0$ for all $v=(1+\langle \omega,v_{x}\rangle,v_{x})$ with $v_{x} \in S\partial N$.
\end{lemma}

\end{document}